\documentclass{amsart}
\usepackage[utf8]{inputenc}
\usepackage{amsthm}
\usepackage[non-compressed-cites, 
non-sorted-cites, alphabetic]{amsrefs}
\usepackage[space]{cite}
\usepackage{fullpage}
\usepackage[margin=1in]{geometry}
\usepackage{amsmath}
\usepackage{stackengine}
\usepackage{mathrsfs} 
\usepackage{amssymb}
\usepackage{enumitem}
 \usepackage{amsfonts}
 \usepackage{bbm} % for \mathbbm{1}
 \usepackage{physics} % for qand
 \usepackage{mathtools}
\usepackage{wasysym}
\usepackage{xcolor}

\definecolor{winered}{rgb}{0.7,0,0} % Peter Hintz's formatting colors
\definecolor{lessblue}{rgb}{0,0,0.7} % Peter Hintz's formatting colors
\usepackage[pdftex,colorlinks=true,linkcolor=winered,citecolor=lessblue,breaklinks=true,bookmarksopen=true]{hyperref}
\usepackage{mathtools}
\definecolor{mypink1}{rgb}{0.858, 0.188, 0.478}

\numberwithin{equation}{section}
\theoremstyle{plain}
\newtheorem{Th}{Theorem}[section]
\newtheorem{Lemma}[Th]{Lemma}
\newtheorem{Cor}[Th]{Corollary}
\newtheorem{prop}[Th]{Proposition}
\newtheorem{definition}[Th]{Definition}
\theoremstyle{definition}
\newtheorem{remarkx}[Th]{Remark}
\newenvironment{remark}
  {\pushQED{\qed}\remarkx}
  {\popQED\endremarkx}
 \newcommand{\R}{\mathbb{R}}

\newcommand{\supp}{\operatorname{supp }}
\newcommand{\lecal}{\mathcal{L}\mathcal{E}}
\renewcommand{\norm}[1]{\left\lVert #1 \right\rVert}
\newcommand{\vertiii}[1]{{\left\vert\kern-0.25ex\left\vert\kern-0.25ex\left\vert #1 
    \right\vert\kern-0.25ex\right\vert\kern-0.25ex\right\vert}}
\newcommand{\inprod}[1]{\left\langle #1\right\rangle}

\allowdisplaybreaks

\title{Integrated Local Energy Decay for the Damped Wave Equation on Stationary Space-Times}
\author{Collin Kofroth}
\address{Department of Mathematics, University of North Carolina, Chapel Hill}
\email{ckofroth@asu.edu}
\thanks{\textit{Funding}. The author gratefully acknowledges the partial support afforded by the F. Ivy Carroll Summer Research Fellowship \newline\indent awarded by the UNC-CH Graduate School and partial funding from NSF grant DMS-2054910 (PI: Jason Metcalfe).}
\date{\today}
\begin{document}
\begin{abstract}
We prove integrated local energy decay for the damped wave equation on stationary, asymptotically flat space-times in $(1+3)$ dimensions. Local energy decay constitutes a powerful tool in the study of dispersive partial differential equations on such geometric backgrounds. By utilizing the geometric control condition to handle trapped trajectories, we are able to recover high frequency estimates without any loss. We may then apply known estimates from the work of Metcalfe, Sterbenz, and Tataru in the medium and low frequency regimes in order to establish local energy decay. This generalizes the integrated version of results established by Bouclet and Royer from the setting of asymptotically Euclidean manifolds to the full Lorentzian case. 
\end{abstract}
\maketitle
\tableofcontents
\section{Introduction}
\subsection{Background}
The goal of this paper is to establish local energy decay for the damped wave equation on asymptotically flat space-times with time-independent metrics (in the sense that $\partial_t$ is a Killing field) subject to the geometric control condition. The primary advance in this work is recovering the high frequency local energy estimate present in \cites{mst20} for waves on non-trapping space-times. Since the aforementioned work only utilized the non-trapping assumption at high frequencies, this establishes the key step in extending time-integrated versions of previously-known results for damped waves on product manifolds (see \cites{BR14}) to the full Lorentzian case. From the proven high frequency estimate, we may apply known results in \cites{mst20} to conclude local energy decay to complete this extension.

Local energy estimates are a collection of rich and well-studied quantities within the field of dispersive partial differential equations, originally introduced on Minkowski space in classical works such as \cites{Mora66,Mora68,Mora75}, \cites{MRS77}. A particularly important class of local energy estimates are the \textit{integrated local energy estimates}; if $u$ solves the homogeneous flat wave equation $$(\partial_t^2-\Delta) u=0, \qquad \Delta=\sum\limits_{j=1}^n\partial_{x_j}^2$$ in spatial dimension $n\geq 3$, then the integrated local energy estimate which we are interested in takes the form \begin{align}\label{LED flat}
\sup_{j\geq 0}\left(\norm{\langle x\rangle^{-1/2}\partial u}_{L^2_tL^2_x\big(\R_+\times\{\langle x\rangle\approx 2^j\}\big)}+\norm{\langle x\rangle^{-3/2}u}_{L^2_tL^2_x\big(\R_+\times\{\langle x\rangle\approx 2^j\}\big)}\right)\lesssim \norm{\partial u(0)}_{L^2},
\end{align} where $\partial=(\partial_t,\nabla)$ denotes the space-time gradient, and $\inprod{x}=(1+|x|^2)^{1/2}$ denotes the Japanese bracket of $x$. This estimate is known to hold in the flat setting through a positive commutator argument using the multiplier introduced in the appendix of \cites{SR05}. In such a case, we will say that \textit{(integrated) local energy decay} holds. This is a quantitative statement on \textit{dispersion}, and it heuristically expresses that the energy of the wave must decay quickly enough within compact spatial sets to be integrable in time. Estimates of this form have significant utility, as they have been used to prove other important measures of dispersion such as Strichartz estimates (see \cites{BT07,BT08}, \cites{JSS90,JSS91}, \cites{MMT08}, \cites{MMTT10},  \cites{MT09,MT12}, \cites{Tat08}, \cites{Toh12},  and the references therein) and pointwise decay estimates (see \cites{Hi22}, \cites{Looi21}, \cites{MTT12}, \cites{Morg20}, \cites{MW21}, \cites{Tat13},  and references in these works). Additionally, local energy estimates have applications to nonlinear wave equations where one can develop estimates on an appropriate linearization of the problem, viewing the nonlinearity as a perturbation. These techniques have been applied in many works; see e.g. \cites{BH10}, \cites{KSS02,KSS04}, \cites{MS06,MS07}, \cites{SW10}, and the citations contained in them. We will be focused on establishing local energy decay rather than demonstrating its utility via applications. 

In \cites{mst20}, the authors proved that local energy decay holds for a broad class of stationary wave operators if and only if
\begin{enumerate}
    \item \textit{The space-time is non-trapping}: There are no null bicharacteristic rays which stay within a compact set for all time. 
    \item \textit{The operator satisfies certain spectral assumptions}: Upon replacing time derivatives in the wave operator with a complex parameter, one requires that this family of operators have no eigenvalues in the lower half-plane nor real resonances/embedded eigenvalues (see \cites{mst20} for more precise definitions); equivalently, one requires analytic continuation of the inverse (\textit{resolvent}) of this family of operators to the entire lower half plane and continuous extension to the real line.
\end{enumerate} They also established results for \textit{almost} stationary operators, though that is not the context of the work presented here. While the authors employed a non-trapping hypothesis, their work did not require product structure on their space-times, which makes their work highly influential in our own.

Although the absence of trapping is known to be necessary for waves to experience local energy decay (see \cites{Ral69}, \cites{Sb15}), one can recover weak local energy decay estimates with a prescribed loss at high frequencies for certain types of trapping (see \cites{Burq98}, \cites{Chr08}, \cites{Ika82,Ika88}, \cites{MMTT10}, \cites{NZ09}, \cites{Toh12}, \cites{WZ11}, and the contained references). When the trapping is sufficiently weak/unstable, then this loss is nominal (in fact, logarithmic); this is the case for both the \textit{Schwarzschild} (\cites{MMTT10}) and \textit{Kerr} (\cites{Toh12}) space-times. Both space-times possess non-trivial trapped sets, which constitute regions where light remains for all time. Although one can extract weak local energy decay estimates, the trapping still generates an immutable barrier to full local energy decay. We will not be working in a scenario that generates loss, although we would be remiss if we did not briefly mention weak local energy decay and essential space-times that enjoy it.

The study of damped waves also possesses a deep history, especially on compact manifolds. The seminal work \cites{RT74} introduced the \textit{geometric control condition}, which required that all null bicharacteristic rays intersect the damping region, and they used it to show that the energy of solutions to damped hyperbolic equations on compact product manifolds enjoys exponential decay in time. The uniform exponential bound is equivalent to so-called \textit{strong stabilization}, whereby one can bound the energy at an arbitrary time by the initial energy multiplied by a monotone-decreasing, non-negative function tending to zero as $t\rightarrow\infty$. This established the sufficiency of geometric control for strong stabilization in such settings, while \cites{Ral69} demonstrated necessity (also, see \cites{Leb96}). The work \cites{BLR92} showed sufficiency for observability and control on compact manifolds with boundary where the observability/control region is contained within the boundary. While there is notably less literature in the non-compact setting, it was proven in \cites{BR14} that local energy decay holds for the damped wave equation on asymptotically \textit{Euclidean} space-times with time-independent metrics under the assumption of geometric control on trapped geodesics. The authors proved dissipative Mourre estimates to obtain uniform resolvent bounds in different frequency regimes in order to apply a limiting absorption argument. This approach is highly dependent on the metric coefficients being independent of time and the product structure (asymptotically Euclidean metrics contain no metric cross terms). This result was improved in \cites{R2018} to estimates in the (weighted) energy space.

In this article, we combine the approaches of \cites{BR14} and \cites{mst20} to establish high frequency local energy estimates for damped waves on stationary, asymptotically flat space-times and explain how such a result can be readily combined with the existing work in \cites{mst20} to prove local energy decay. We underscore that we are not requiring the product structure evident in \cites{BR14} nor \cites{R2018} but, instead, allow for the full Lorentzian formulation.  Non-product metrics possess non-trivial cross terms and are called \textit{non-static}, of which the Kerr metric constitutes an important example. We most closely keep to the framework present in \cites{mst20}, which does not assume product structure and has results for even more general asymptotically flat non-trapping space-times (such as non-stationary ones). We again stress their use of a non-trapping hypothesis, which we replace by imposing geometric control. Trapping is an intrinsically high-frequency phenomenon, so only their high frequency work is affected by the trapping. Hence, this is the portion of the argument that needs modification to ensure local energy decay, and this is where the influence of \cites{BR14} comes into play. Since the medium and low frequency analyses (as well as the procedure of combining the different frequency regime estimates into the full local energy decay estimate) do not depend on the non-trapping hypothesis nor use the damping themselves, the corresponding results in \cites{mst20} readily apply (i.e. our problem essentially becomes a special case here). We omit the details of such results in this work, but we will explain why they apply in our context. 

\subsection{Problem Setup and Main Results}\label{setup}
Let $(\R^{4},g)$ be a Lorentzian manifold with coordinates $(t,x)\in\R\times\R^3,$ where $g$ has signature $(-+++).$ We will consider \textit{damped wave operators} of the form $$P=\Box_g+iaD_t,\qquad \Box_g=D_\alpha g^{\alpha\beta}D_\beta,$$ where $a\in C_c^\infty(\R^3)$ is non-negative and positive on an open set, and $D_\alpha=\frac{1}{i}\partial_\alpha,\ \alpha=0,1,2,3$. Greek indices will generally range over such values, whereas Latin indices will run over the integers $1,2,$ and $3$. Notice that we are using the standard Einstein summation convention, which we will do throughout this work. We will also subject $g$ to an \textit{asymptotic flatness} condition. More precisely, we first define the norm $$\norm{h}_{AF}=\sum\limits_{|\alpha|\leq 2} \norm{\inprod{x}^{|\alpha|}\partial^\alpha h}_{\ell^1_j L^\infty([0,T]\times A_j)},$$
where  $A_j=\{\inprod{x}\approx 2^j\}$ for $j\geq 0$ denote inhomogeneous dyadic regions, and $\ell^1_j$ denotes the $\ell^1$ norm over the $j$ index. The notation $A\lesssim B$ means that $A\leq CB$ for some $C>0,$ and the notation $A\approx B$ means that $B\lesssim A\lesssim B$. In the definition of the $A_j$'s, we require that these implicit constants are compatible to cover $\R^3.$ 
This allows us to define the $AF$ topology. 
\begin{definition}
    We say that $P$ is \textit{asymptotically flat} if 
    $\norm{g-m}_{AF}<\infty$,
    where $m$ denotes the Minkowski metric, and  
$$\norm{\inprod{x}^{|\alpha|}\partial^\alpha g}_{\ell_j^1L^\infty([0,T]\times A_j)}\lesssim_{\alpha} 1$$ for all $\alpha\in\mathbb{N}^3$ with $|\alpha|\geq 3.$ 
\end{definition} The latter condition will be necessary for certain functions appearing in this work to be symbolic in the Kohn-Nirenberg sense. We remark that the dyadic summability assumptions on our metric are weaker than the long-range perturbation condition present in \cites{BR14} (which provides a symbolic-type decay estimate for derivatives of the metric in $x$ in terms of $\inprod{x}^{-\rho},$ with $\rho>0$ fixed). In \cites{BR14}, the damping is not assumed to be compactly-supported, but rather non-negative everywhere and subject to a similar symbolic estimate to the metric (with an additional power of decay). Since the damping is a \textit{helpful} term which will only be necessary within a compact spatial set (this is made explicit with the introduction of the parameter $R_0$ on the next page), it is unnecessarily beneficial for us to assume that is it non-negative everywhere.

We will primarily be interested in when $\partial_t$ is a Killing field for $g$, in which case we say that $P$ is \textit{stationary}.

\begin{definition}
We say that $P$ is stationary if $g$ is independent of $t$.
\end{definition}
Next, we introduce \begin{itemize}
    \item the parameters $R_0$ and $\textbf{c}$, which are such that \begin{align*}
       \norm{g-m}_{AF_{>R_0}}&\leq \textbf{c}\ll 1, 
        \end{align*} where the subscript denotes the restriction of the norm to $\{|x|>R_0\}.$ The parameter $\textbf{c}$ should be viewed as being fixed first, after which we find an $R_0$ for which the above holds. Without loss of generality, we will assume that $\supp a\subset \{|x|\leq R_0\}$ (as it is unnecessarily beneficial outside of this set).
            \item the sequence  $(c_j)_{j\geq \log_2 R_0}$ satisfying
        \begin{align*}
      \norm{g-m}_{AF(A_j)}\lesssim c_j,\qquad
        \sum\limits_j c_j\lesssim \textbf{c}, 
    \end{align*} where $\norm{\cdot}_{AF(A_j)}$ denotes the restriction of the $AF$ norm to the dyadic region $A_j.$ We may assume, without any loss of generality, that the sequence is slowly-varying, i.e. $$c_j/c_k\leq 2^{\delta|k-j|},\qquad\delta\ll 1.$$ This sequence will be utilized when working in spatial weights within dyadic regions.
    \end{itemize}
These parameters tell us that, outside of a large enough spatial ball, the operator $P$ is a uniformly small perturbation of the flat wave operator $\Box_m=\partial_t^2-\Delta$ (which we simply denote as $\Box$). The sequence $(c_j)$ provides a quantitative measure on the size of the $AF$ norm throughout each spatial dyadic region outside of this ball.

We will also assume throughout that the vector field $\partial_t$ is uniformly time-like, which essentially constitutes a choice of coordinates. This condition, coupled with the signature of the metric, ensures that $D_ig^{ij}D_j$ is uniformly elliptic, i.e. \begin{equation}\label{ellipticity}
    g^{ij}\xi_i\xi_j\approx |\xi|^2,\qquad \xi\neq 0,
\end{equation}
where $|\cdot|$ denotes the standard Euclidean norm. This follows from the positive-definiteness of the momentum-energy tensor when applied to time-like vector fields.
 
 Next, we define the local energy norms \begin{align*}\norm{u}_{LE}&=\sup_{j\geq 0}\norm{\langle x\rangle^{-1/2} u}_{L^2_tL^2_x\big(\R_+\times A_j\big)},\\
\norm{u}_{LE^1}&=\norm{\partial u}_{LE}+\norm{\langle x\rangle^{-1}u}_{LE}.
\end{align*} A predual-type norm to the $LE$ norm is the $LE^*$ norm, which is defined as $$ \norm{f}_{LE^*}=\sum\limits_{j=0}^\infty\norm{\langle x\rangle^{1/2} f}_{L^2_tL^2_x\big(\R_+\times A_j\big)}.$$ Here, $L^p_tL^q_x$ denotes the Bochner space $L^p(\R_+, L^q(\R^3)).$ In the particular case of $p,q=2,$ then this is a Hilbert space; we will use $\inprod{\cdot,\cdot}$ to denote the inner product on $L^2_tL^2_x$. Lastly, we define the sum-space norm $$\norm{f}_{LE^*+L^1_tL^2_x}=\inf_{f=f_1+ f_2}\left(\norm{f_1}_{LE^*}+ \norm{f_2}_{L^1_tL^2_x}\right).$$

If we wish for the time interval to be e.g. $[0,T]$ in the above norms, then we will use the notation $\norm{u}_{LE[0,T]},\norm{u}_{LE^1[0,T]},\norm{u}_{LE^*[0,T]}, \norm{u}_{LE^*+L^1_tL^2_x[0,T]}$ (although we will write $LE^*[0,T]+L^1_tL^2_x[0,T]$ when referring to this space outside of norm subscripts), etc. A subscript of $c$ on any of these spaces denotes compact spatial support.

There are two additional function spaces that will be utilized extensively in this work. The first is the class of Schwartz functions $\mathcal{S}(\R^4)$, which will be useful for approximation arguments. The second is a particular collection of functions which is often the natural class to study wave equations. 
\begin{definition}
Let $T>0$. We define the class $\mathcal{W}_T$ to be the space of all functions $u\in C^2([0,T]\times\R^3)$ for which there exists $R>0$ so that $u(t,x)=0$ for all $t\in [0,T]$ and $|x|>R$. That is, $$\mathcal{W}_T=\{u\in C^2([0,T]\times\R^3): (\exists R>0)(\forall |x|>R)(\forall t\in [0,T])\ \ u(t,x)=0\}.$$
\end{definition}
We are interested in Cauchy problems of the form
\[\left\{\begin{aligned}\label{cauchy prob}
Pu&=f\in LE^*[0,T]+L^1_t L^2_x [0,T],\\
u[0]&=(u(0),\partial_t u(0))\in \dot{H}^1\oplus L^2.
\end{aligned}\right.\]
\begin{remark} 
The decay conditions on $u\in\mathcal{W}_T$ are not as restrictive as they might initially appear. If the Cauchy data is compactly-supported, then the condition is free by finite speed of propagation. If it is not, then one can approximate the data (which generically lives in the energy space) by compactly-supported data. The regularity conditions on $u$ are also not restrictive, as one can perform density arguments to reduce to the case of increased regularity.
\end{remark}

Now, we state the pertinent local energy estimates for such problems.\begin{definition}
We say that \textit{(integrated) local energy decay} holds for an asymptotically flat wave operator if the following estimate holds for all $T>0$: \begin{equation}\label{LED}
\norm{u}_{LE^1[0,T]}+\norm{\partial u}_{L^\infty_t L^2_x[0,T]}\lesssim\norm{\partial u(0)}_{L^2}+\norm{Pu}_{LE^*+L^1_tL^2_x[0,T]}\end{equation} for all $u\in \mathcal{W}_T$ such that $u[0]\in \dot{H}^1\oplus L^2$, with the implicit constant being independent of $T$.
\end{definition} 
The notion of an asymptotically flat wave operator is more broad than an asymptotically flat damped wave operator. They need not feature a damping term, and they are allowed to possess general lower-order terms which are asymptotically flat in an appropriate sense (see Definition 1.1 in \cites{mst20} for a precise definition).

Note that, due to global energy conservation for the flat wave problem, the general definition of local energy decay that we have given here is consistent with the integrated local energy  estimate for the flat wave equation in (\ref{LED flat}) (in the inhomogeneous case, one applies H\"older's inequality to the forcing). This estimate is known to hold whenever $P$ is a small asymptotically flat perturbation of $\Box$ (see \cites{Ali06}, \cites{MS06, MS07},  \cites{MT12}). In \cites{mst20}, the authors considered large $AF$ perturbations and proved that, for stationary problems, the local energy decay estimate (\ref{LED}) is equivalent to assuming that the wave operator $P$ is non-trapping and has no negative eigenfunctions ($L^2$ eigenfuctions with corresponding eigenvalues in the lower half-plane) nor real resonant states (outgoing non-$L^2$ eigenfunctions with real ``eigenvalues,'' which are called resonances); see Definitions 2.2, 2.4, and 2.8 in \cites{mst20} for more precise statements. The non-trapping hypothesis only arose during their proof of a high frequency estimate (Theorem 2.11 in \cites{mst20}), which took the form \begin{align}\label{high freq LED}
\norm{u}_{LE^1[0,T]}+\norm{\partial u}_{L^\infty_t L^2_x[0,T]}\lesssim \norm{\partial u(0)}_{L^2}+\norm{\inprod{x}^{-2} u}_{LE[0,T]}+\norm{Pu}_{LE^*+L^1_tL^2_x[0,T]}.
\end{align}
The implicit constant in the above estimate is crucially independent of $T$. This estimate does not require $u$ to be truncated to large time frequencies, but this is the context in which it is used in proving local energy decay.

The added spatial weight in the error term does not play a particular role in making this high frequency. Rather, it is the weight that naturally arises when performing a bootstrapping argument in the proof of the estimate; it is largely unimportant since this estimate can be reduced to studying solutions with compact spatial support (see Section \ref{reductions section}).

\begin{remark}\label{high freq rem} To see this as an estimate on the high frequencies, let $u\in\mathcal{S}(\R^4)$ be frequency-supported in time for $\tau$ in the range $1\ll \tau_1\leq |\tau|<\infty$ ((\ref{high freq LED}) is only applied for such $u$ in the proof of local energy decay). Then, we can use Plancherel's theorem in $t$ to obtain that 
\begin{align} \label{highfreqtrunc}
    \norm{\inprod{x}^{-2}u}_{LE_{t,x}}
    &\approx\norm{\inprod{x}^{-2}\hat{u}(\tau,x)}_{LE_{\tau,x}}\lesssim \frac{1}{\tau_1}\norm{\inprod{x}^{-2}\tau\hat{u}(\tau,x)}_{LE_{\tau,x}}
   \lesssim \frac{1}{\tau_1} \norm{u}_{LE^1_{t,x}}. 
\end{align} For large enough $\tau_1,$ this term can be absorbed into the left-hand side of (\ref{high freq LED}), providing local energy decay for solutions restricted to high frequencies. 

In fact, we may apply the high frequency estimate (\ref{high freq LED}) to $u(t-T/2)$ to get (after dropping the uniform energy piece) that
$$\norm{u}_{LE^1[-T/2,T/2]}\lesssim\norm{\partial u(-T/2)}_{L^2}+\norm{\inprod{x}^{-2}u}_{LE[-T/2,T/2]}+\norm{Pu}_{LE^*[-T/2,T/2]}.$$ Since the implicit constant is independent of $T$, we may take the limit as $T\rightarrow\infty$ and apply the prior work in (\ref{highfreqtrunc}) to obtain that 
$$\norm{u}_{LE^1}\lesssim\norm{Pu}_{LE^*}.$$ This is the context that we will apply the estimate to establish local energy decay.
\end{remark}

Our first main theorem is the following, which states that we recover the high frequency estimate (\ref{high freq LED}) of \cites{mst20} when working with damped waves and replacing the non-trapping hypothesis with the geometric control condition.
\begin{Th}\label{high freq est}  
Let $P$ be a stationary, asymptotically flat damped wave operator satisfying the geometric control condition, and suppose that $\partial_t$ is uniformly time-like. Then, the high frequency local energy estimate (\ref{high freq LED}) holds for all for all $u\in \mathcal{W}_T$ such that $u[0]\in\dot{H}^1\oplus L^2$. The implicit constant is independent of $T$.
\end{Th}
 The geometric control condition, initially introduced in \cites{RT74} for dissipative hyperbolic equations on compact product manifolds, requires that every trapped null bicharacteristic ray intersects the damping region. We will make this more precise in Section \ref{dynamical frame}.
  \begin{remark}The implicit constant in the bound depends on $R_0$. In fact, much of our work will implicitly depend on $R_0$ due to our applications of asymptotic flatness. It is essential to note that this parameter is fixed second (with $\textbf{c}$ being fixed first), after which our other parameters (such as the scaling parameter $\gamma$ and the high-frequency parameter $\lambda$ which will both be introduced in Chapter 3) will be chosen (and hence depend on it). We will not track the dependence on $R_0$ within our implicit constants any longer. Our constants throughout will \textit{not} depend on $T$, however.
  \end{remark}
  Our second main theorem is local energy decay.
\begin{Th}\label{LED thm}
Let $P$ be a stationary, asymptotically flat damped wave operator satisfying the geometric control condition, and suppose that $\partial_t$ is uniformly time-like. Then, local energy decay
holds, with the implicit constant in (\ref{LED}) independent of $T$.
\end{Th}
\begin{remark}
    Unlike \cites{mst20}, we do not require any spectral hypotheses. The damping eliminates the possibility of non-zero real resonances, and a zero resonance cannot occur since $ P|_{D_t=0}$ is elliptic, which is independent of the damping. The latter is discussed further in Section \ref{low freq}.
\end{remark}
This follows rather directly from our high frequency estimate and the existing work in \cites{mst20}. We will cite the necessary result from \cites{mst20} and explain why they apply here.

The structure of the paper is as follows. In Section \ref{dynamical frame}, we will introduce the Hamiltonian formalism required to define trapping and geometric control, then we will state a key lemma (Lemma \ref{escape}) for the proof of Theorem \ref{high freq est}; namely, we construct an appropriate escape function and lower-order correction to allow for a positive commutator argument proof of the theorem. In Section \ref{flow results section}, we establish various results on the bicharacteristic flow that are vital for proving Lemma \ref{escape}, which we establish in Section \ref{escape section}. In Section \ref{reductions section}, we will establish and cite supplemental energy estimates and provide multiple case reductions to simplify the proof of Theorem \ref{high freq est}. In Section \ref{thm section}, we will prove Theorem \ref{high freq est}. Finally, Section \ref{LED section} will present the applicable theorems from \cites{mst20} which are required to establish local energy decay for our problem and discuss why they apply in our setting. Sections \ref{med freq} and \ref{low freq} introduce the relevant medium and low frequency estimates, respectively, and Section \ref{ext arg} provides a discussion on the proof of Theorem \ref{LED thm}.
\subsection{Cutoff Notation}
For the remainder of paper, we will fix the cutoffs
$$\chi\in C_c^\infty\text{ non-increasing},\ \chi\equiv 1\text{ for } |x|\leq 1,\ \chi\equiv 0\text{ for } |x|> 2,$$  $$\chi_{< R}(|x|)=\chi\left(\frac{|x|}{R}\right),\qquad \chi_{>R}=1-\chi_{< R},$$ and $$\chi_R\in  C_c^\infty,\ 0\leq \chi_R\leq 1,\ \supp\chi_R\subset\{|x|\approx R\}.$$ We will choose $\chi$ such that its square root is smooth (otherwise, we replace $\chi$ with $\chi^2$; we only use $\chi$ for notational convenience). When working in frequency variables, we will often add the variable into the subscript to make the dependence clear (e.g. $\chi_{|\xi|>\lambda}$). We will also occasionally write $r=|x|.$
\subsection{Acknowledgements}
This work was performed as a portion of the Ph.D. dissertation of the author at the University of North Carolina at Chapel Hill (see \cite{Kof22}). The author would like to thank his advisor, Jason Metcalfe, for suggesting this problem, engaging in countless helpful conversations, and carefully reading multiple drafts of this work. He would also like to thank Yaiza Canzani, Hans Christianson, Jeremy Marzuola, and both referees of this article for providing many beneficial comments and references.

\section{High Frequency Analysis}
\subsection{Introduction}
In this section, we will establish Theorem \ref{high freq est}. The notions of trapping and geometric control are intrinsically dynamical, so we will provide a thorough discussion of the relevant theory. Namely, we must introduce the bicharacteristic flow generated by the principal symbol of the damped wave operator and the properties that it satisfies. From here, we will construct an escape function and correction term in order to apply a utilize commutator argument and prove the theorem. The constructed symbols will satisfy an appropriate positivity bound, which will allow us to apply the sharp G\r{a}rding inequality upon swapping to the framework of pseudodifferential operators. The symbols will also be supported in an unbounded range of frequencies $[\lambda,\infty$) with $\lambda\gg 1$. This will be fundamentally important for bootstrapping error terms resulting from employment of pseudodifferential calculus.
\subsection{Dynamical Framework}\label{dynamical frame}
 In order to state the geometric control condition more precisely, we must first outline our dynamical framework, which is rooted in the Hamiltonian dynamics pertaining to the principal symbol of the operator $P$. Since we assumed that $\partial_t$ is uniformly time-like, the signature of the metric and the cofactor expansion for the inverse metric tell us that $g^{00}\lesssim -1$, as well. This allows us to divide through by $-g^{00}$ and preserve the assumptions on the operator coefficients (see \cites{MT12}). Hence, we may assume (without loss of generality) that $g^{00}=-1$. 

After these modifications, the principal symbol of $P$ is $$p(\tau, x,\xi)=-(\tau^2-2\tau g^{0j}(x)\xi_j-g^{ij}(x)\xi_i\xi_j).$$ This is considered as a smooth function on $T^*\R^4\setminus o$, where $o$ denotes the zero section. This symbol generates a bicharacteristic/Hamiltonian flow on $\R\times T^*\R^4$ given by  $\varphi_s(w)=\left(t_s(w),\tau_s(w), x_s (w),\xi_s(w)\right)$
which solves  
$$\left\{
\arraycolsep=1pt
\def\arraystretch{1.2}
\begin{array}{rlrl}
\dot{t}_s&=\partial_\tau p(\varphi_s(w)),& \qquad
\dot{\tau}_s&=-\partial_t p(\varphi_s(w)),\\
\dot{x}_s&=\nabla_\xi p(\varphi_s(w)),& \qquad
\dot{\xi}_s&=-\nabla_x p(\varphi_s(w)),
\end{array}\right.$$
with initial data $w\in T^*\R^4.$  Explicitly, one can write the system as
$$\left\{
\arraycolsep=1pt
\def\arraystretch{1.2} % these commands seem to give equals signs extra space sporadically, fix when visible by replacing = with {\,=\,}
\begin{array}{rlrl}
\dot{t}_s&=-2\tau_s+2g^{0j}(x_s)[\xi_s]_j, & 
\dot{\tau}_s&=0,\\
(\dot{x}_s)_k&=2\tau_sg^{0k}(x_s)+2g^{kj}(x_s)[\xi_s]_j,& \qquad
(\dot{\xi}_s)_k&=-2\tau_s\partial_{x_k} g^{0j}(x_s)[\xi_s]_j-\partial_{x_k} g^{ij}(x_s)[\xi_s]_i[\xi_s]_j.
\end{array}\right.$$

Since $g$ is smooth and asymptotically flat, and $\partial_t$ is uniformly time-like, we have a unique, smooth, globally-defined flow with smooth dependence on the data.  We will have particular interest in \textit{null} bicharacteristics, i.e. those with initial data lying in the zero set of $p$ (also called the \textit{characteristic set} of $P$ and denoted $\text{Char}(P)$). Using the flow $\varphi_s$, we define the \textit{forward} and \textit{backward trapped} and \textit{non-trapped} sets with respect to $\varphi_s$, respectively, as
\begin{align*}
\Gamma_{tr}&=\left\lbrace w\in T^*\R^4\setminus o:
\sup_{s\geq 0} |x_{ s}(w)|<\infty\right\rbrace \cap \text{Char}(P),\\
\Lambda_{tr}&=\left\lbrace w\in T^*\R^4\setminus o:
\sup_{s\geq 0} |x_{- s}(w)|<\infty\right\rbrace \cap \text{Char}(P),\\
\Gamma_{\infty}&=\left\lbrace w\in T^*\R^4\setminus o:
 |x_{s}(w)|\rightarrow\infty\text{ as }s\rightarrow\infty\right\rbrace \cap \text{Char}(P),\\
\Lambda_{\infty}&=\left\lbrace w\in T^*\R^4\setminus o:
 |x_{-s}(w)|\rightarrow\infty \text{ as }s\rightarrow\infty\right\rbrace \cap \text{Char}(P).
\end{align*} 

The \textit{trapped} and \textit{non-trapped} sets are defined as 
$$\Omega^p_{tr}=\Gamma_{tr}\cap\Lambda_{tr}\qand\Omega^p_{\infty}=\Gamma_{\infty}\cap\Lambda_{\infty},$$ respectively. 
\begin{definition}
    The flow is said to be  \textit{non-trapping} if $\Omega^p_{tr}{\,=\,}\emptyset$.
\end{definition} 
Now, we may state the {geometric control condition} precisely. Recall that our damping function was denoted $a$.
\begin{definition}
    We say that geometric control holds if \begin{align}\label{GCC}
    (\forall w\in \Omega^p_{tr})(\exists s\in\R)\ \ a(x_s(w))>0.
    \end{align}
\end{definition} In contrast to the definition in \cites{RT74} (given in Assumption (A)), we apply this condition specifically to the trapped null bicharacteristic rays (since \textit{all} null bicharacteristic rays are trapped when the manifold is compact, such a specification was unnecessary in \cites{RT74}). We will assume that (\ref{GCC}) holds. Note that if (\ref{GCC}) holds and $a\equiv 0$, then $\Omega^p_{tr}$ must be empty, meaning that the flow is non-trapping. In this case, we are back in the setting of \cites{mst20}. For this reason, we will assume that $a>0$ on an open set.

It will be beneficial to utilize a scaling property of $P$. Given a solution $u$ to $Pu=f$, consider  $$\tilde{v}(t,x):=\gamma^{-2} u(\gamma t,\gamma x),\qquad\gamma>0.$$ If we call $$\tilde{P}=D_\alpha \tilde{g}^{\alpha\beta}D_\beta+i\gamma \tilde{a} D_t,\qquad \tilde{g}^{\alpha\beta}(x)=g^{\alpha\beta}(\gamma x),\qquad \tilde{a}(x)=a(\gamma x),$$ then $\tilde{v}$ solves  $$\tilde{P} \tilde{v}=\tilde{f},\qquad \tilde{f}(t,x)=f(\gamma t,\gamma x)$$ if and only if $u$ solves  $Pu=f$ (we can similarly undo the scaling to move between the frameworks). Notice that the scaled problem allows for an arbitrarily large constant $\gamma$ in front of the damping.

Analogous Hamiltonian systems and trapped sets exist for the principal symbol $\tilde{p}$ of $\tilde{P}$, and this amounts to simply replacing $g$ by $\tilde{g}$. If we assume that geometric control holds for the flow generated by $p$, then we must check that it holds for the scaled problem.
\begin{prop}
Assume that (\ref{GCC}) holds. Then, for any $\gamma>0$, (\ref{GCC}) holds for the flow generated by $\tilde{p}$, with $a$ replaced by $\tilde{a}$.
\end{prop}
Note that since $g^{00}\equiv -1,$ it follows that $\tilde{g}^{00}\equiv -1.$
\begin{proof}
The flow generated by $\tilde{p}$ solves the system
$$\left\{\begin{aligned}
    \frac{d}{ds} \tilde{t}_s&=-2\tilde{\tau}_s+2\tilde{g}^{0j}(\tilde{x}_s)[\tilde{\xi}_s]_j,\\
    \frac{d}{ds}\tilde{\tau}_s&=0,\\
    \frac{d}{ds}(\tilde{x}_s)_k&=2\tilde{\tau}_s \tilde{g}^{0k}(\tilde{x}_s) +2\tilde{g}^{kj}(\tilde{x}_s)[\tilde{\xi}_s]_j,\\
    \frac{d}{ds}(\tilde{\xi}_s)_k&=
    -2\tilde{\tau}_s\partial_{x_k}\tilde{g}^{0j}(\tilde{x}_s)[\tilde{\xi}_s]_j-\partial_{x_k}\tilde{g}^{ij}(\tilde{x}_s)[\tilde{\xi}_s]_i[\tilde{\xi}_s]_j, \\
    (\tilde{t}_s,\tilde{\tau}_s,\tilde{x}&_s,\tilde{\xi}_s)\big|_{s=0}=(t,\tau,x,\xi).
\end{aligned}\right.$$
Applying the chain rule and multiplying through by $\gamma$ provides us with the system
$$\left\{\begin{aligned}
    \frac{d}{ds} (\gamma \tilde{t}_s)&=-2(\gamma\tilde{\tau}_s)+2{g}^{0j}(\gamma \tilde{x}_s)[\gamma\tilde{\xi}_s]_j,\\
    \frac{d}{ds}(\gamma \tilde{\tau}_s)&=0,\\
    \frac{d}{ds}(\gamma \tilde{x}_s)_k&=2(\gamma \tilde{\tau}_s) {g}^{0k}(\gamma \tilde{x}_s) +2{g}^{kj}(\gamma \tilde{x}_s)[\gamma\tilde{\xi}_s]_j,\\
    \frac{d}{ds}(\gamma\tilde{\xi}_s)_k&=-2(\gamma \tilde{\tau}_s)[(\partial_{x_k}{g}^{0j})(\gamma \tilde{x}_s)][\gamma\tilde{\xi}_s]_j-[(\partial_{x_k}{g}^{ij})(\gamma \tilde{x}_s)][\gamma\tilde{\xi}_s]_i[\gamma\tilde{\xi}_s]_j, \\
    \big((\gamma \tilde{t})_s,(\gamma &\tilde{\tau})_s,(\gamma \tilde{x})_s,(\gamma\tilde{\xi})_s\big)\big|_{s=0}=(\gamma t,\gamma \tau,\gamma x,\gamma \xi).
\end{aligned}\right.$$
This is the same system that is solved by the Hamiltonian flow generated by $p$ with initial data $(\gamma t,\gamma \tau,\gamma x,\gamma \xi).$ By uniqueness, we can conclude that 
$$\left\{\arraycolsep=1pt
\begin{array}{rlrl}
\gamma\tilde{t}_s(t,\tau,x,\xi)&=t_s(\gamma t,\gamma \tau,\gamma x,\gamma\xi),& \qquad
\gamma\tilde{\tau}_s(t,\tau, x,\xi)&=\tau_s(\gamma t,\gamma \tau,\gamma x,\gamma\xi)\\
\gamma\tilde{x}_s(t,\tau,x,\xi)&=x_s(\gamma t,\gamma \tau,\gamma x,\gamma\xi),& \qquad
\gamma\tilde{\xi}_s(t,\tau, x,\xi)&=\xi_s(\gamma t,\gamma \tau,\gamma x,\gamma\xi).
\end{array}\right.
$$
Now, let $w=\Omega^{\tilde{p}}_{tr}.$ From the above, we have that $$\tilde{x}_s(w)=\gamma^{-1}x_{s}(\tilde{w}),\qquad \tilde{w}=\gamma w.$$ Since $$\sup_{s\in\R}|x_{s}(\tilde{w})|=\gamma\sup_{s\in\R}|\tilde{x}_s(w)|<\infty,$$ it follows that $\tilde{w}\in\Omega^p_{tr}$. By (\ref{GCC}), there exists $s'\in\R$ so that $a(x_{s'}(\tilde{w}))>0,$ and so $$\tilde{a}(\tilde{x}_{s'}(w))=a(\gamma \tilde{x}_{s'}(w))=a(x_{s'}(\tilde{w}))>0,$$ which completes the proof. 
\end{proof}

Now that we have shown that geometric control is invariant under scaling, we will fix a large $\gamma>0$ and study the problem from the scaled perspective (where our damping is now multiplied by $\gamma$) while reverting back to our original notation ($x$ and $\xi$, no tildes, etc.). More precise conditions on the size of $\gamma$ will come in Section \ref{escape section}.  It is readily seen that it is equivalent to prove Theorem \ref{high freq est} for the scaled problem, where we now have a large constant in front of the damping term. 

Our proof of Theorem \ref{high freq est} is a positive commutator argument. At the symbolic level, this requires the construction of an escape function (as well as a lower-order correction). We must consider the skew-adjoint contribution of $P$, which will be a purely beneficial term due to the presence of the damping. Let $p$ and $s_{skew}$ represent the principal symbols of the self and skew-adjoint parts of $P$, respectively. Namely,
\begin{align*}
    {p}(\tau,x,\xi)&=-(\tau^2-2\tau {g}^{0j}(x)\xi_j-{g}^{ij}(x)\xi_i\xi_j)\\
    {s}_{skew}(\tau,x,\xi)&=i\gamma \tau a (x).
\end{align*} The multiplication by $\gamma$ in $s_{skew}$ will prove advantageous for a bootstrapping argument, which is precisely why we implement the $\gamma$-scaling. Now, we are ready to state our escape function result, which we will prove in Section \ref{escape section}.
\begin{Lemma}\label{escape}
For all $\lambda>1$, there exist symbols $q_j\in S^j(T^*\R^3)$ and $m\in S^0(T^*\R^3)$, all supported in $|\xi|\geq\lambda$, so that $$ H_{{p}}q-2i{s}_{skew}q+{p}m\gtrsim \mathbbm{1}_{|\xi|\geq \lambda}\inprod{x}^{-2}\left(\tau^2+|\xi|^2\right),$$ where $q=\tau q_0+q_1$. \end{Lemma}

Here, $S^m(T^*\R^n)$ denotes the $m^{\operatorname{th}}$-order \textit{Kohn-Nirenberg} symbol class. To each $q\in S^m(T^*\R^n)$, we will associate the pseudodifferential operator $q^{\operatorname{w}}(x,D):\mathcal{S}(\mathbb{R}^n)\rightarrow\mathcal{S}(\mathbb{R}^n),$ namely the \textit{Weyl quantization} of $p$, which is defined via the action $$q^{\operatorname{w}}(x,D)u(x)=(2\pi)^{-n}\int\limits_{\R^n}\int\limits_{\R^n}e^{i(x-y)\cdot\xi}q\left(\frac{x+y}{2},\xi\right)u(y)\, dyd\xi.$$

In the proof of Lemma \ref{escape}, it will be useful to work with the half-wave decomposition, which allows us to avoid the cross terms in the principal symbol. To that end, we factor $p$ as $$p(\tau, x,\xi)=-(\tau-b^+(x,\xi))(\tau-b^-(x,\xi)),$$
where $$b^{\pm}(x,\xi)={g^{0j} \xi_j\pm\sqrt{\left(g^{0j}\xi_j\right)^2+g^{ij}\xi_i\xi_j}}.$$ 

Observe that $b^\pm$ are both homogeneous of degree 1 in $\xi$. Additionally, they are both signed. 
 \begin{prop}\label{b pm sign} For any $(x,\xi)\in T^*\R^3\setminus o$, we have that
 $b^+(x,\xi)>0>b^-(x,\xi).$
 \end{prop}
 \begin{proof}
 Let $\xi\neq 0.$ First, we show that $b^+>b^-$. Indeed, observe that $$b^+-b^-=2\sqrt{\left(g^{0j}\xi_j\right)^2+g^{ij}\xi_i\xi_j}>0$$ using the ellipticity (see \ref{ellipticity}). Using ellipticity again, we have that $$\sqrt{\left(g^{0j}\xi_j\right)^2+g^{ij}\xi_i\xi_j}> |g^{0j}\xi_j| .$$ Thus, $$b^+> g^{0j}\xi_j+ |g^{0j}\xi_j|\geq 0,\qquad b^-< g^{0j}\xi_j- |g^{0j}\xi_j|\leq 0.$$ 
 \end{proof}

We will call  $p^\pm =\tau-b^{\pm},$ so that $p=-p^+p^-.$ In particular, $p=0$ if and only if $p^+=0$ or $p^-=0$; due to Proposition \ref{b pm sign}, it cannot be the case that $p^+(w)=p^-(w)=0$ for any $w\in T^*\R^4\setminus o.$ The Hamiltonians $p^\pm$ also generate flows  
$\varphi_s^\pm(w){\,=\,}\left(t_s^\pm(w),\tau_s^\pm(w), x_s^\pm (w),\xi_s^\pm(w)\right)$ on $\R\times T^*\R^4$ 
which solve the Hamiltonian systems 
$$\left\{
\arraycolsep=1pt
\def\arraystretch{1.2}
\begin{array}{rlrl}
\dot{t}_s^\pm&=\partial_\tau p^{\pm}(\varphi_s^\pm(w)), &
\dot{\tau}_s^\pm&=-\partial_t p^{\pm}(\varphi_s^\pm(w)),\\
\dot{x}_s^\pm&=\nabla_\xi p^{\pm}(\varphi_s^\pm(w)), \qquad&
\dot{\xi}_s^\pm&=-\nabla_x p^{\pm}(\varphi_s^\pm(w)),
\end{array}\right.$$
with initial data $w\in T^*\R^4.$ Note that
$$\left\{
\arraycolsep=1pt
\def\arraystretch{1.2}
\begin{array}{rlrl}
\dot{t}_s^\pm&=1, &
\dot{\tau}_s^\pm&=0,\\
\dot{x}_s^\pm&=-\nabla_\xi b^{\pm}(\varphi_s^\pm(w)),\qquad &
\dot{\xi}_s^\pm&= \nabla_x b^{\pm}(\varphi_s^\pm(w)).
\end{array}\right.$$
 There is a direct correspondence between null bicharacteristics for $\varphi_s$ and null bicharacteristics for $\varphi_s^\pm$.
\begin{prop} \label{bichar equiv}
Every null bicharacteristic for the flow generated by $p$ is a null bicharacteristic for the flow generated by either $p^+$ or $p^-$. The converse is also true.
\end{prop}
\begin{proof}
Recall that for any $(t',\tau',x',\xi'){\,=:\,}w\in T^*\R^4\setminus o,$ we have that $p(w){\,=\,}0$ if and only if either $p^+(w){\,=\,}0$ or $p^-(w){\,=\,}0$. Without loss of generality, suppose that $p^+(w){\,=\,}0.$ The Hamiltonians $p$ and $p^+$ generate the systems
\begin{equation}\label{bichar sys 1}
\left\{
\arraycolsep=1pt
\def\arraystretch{1.2}
\begin{array}{rlrl}
\dot{t}_s&=-p^+(\varphi_s)-p^-(\varphi_s), &
\dot{\tau}_s&=0,\\
(\dot{x}_s)_k&=-p^+(\varphi_s) p^-_{\xi_k}(\varphi_s)-p^-(\varphi_s) p^+_{\xi_k}(\varphi_s), \qquad &
(\dot{\xi}_s)_k&=p^+(\varphi_s) p_{x_k}^-(\varphi_s)+p^-(\varphi_s) p_{x_k}^+(\varphi_s),
\end{array}\right.
\end{equation}
and
\begin{equation}\label{bichar sys 2}\left\{
\arraycolsep=1pt
\def\arraystretch{1.2}
\begin{array}{rlrl}
\dot{t}_s^+&=1, &
\dot{\tau}_s^+&=0,\\
(\dot{x}_s^+)_k&= p_{\xi_k}^{+}(\varphi_s^+),\qquad &
(\dot{\xi}_s^+)_k&=- p_{x_k}^{+}(\varphi_s^+),
\end{array}\right.\end{equation}
respectively. We will take both systems to have initial data $w.$

We claim that since $p^+(w){\,=\,}0,$ we must have that $p^+(\varphi_s(w)){\,=\,}0$ for all $s$. If not, then there would exist $s'$ so that $p^-(\varphi_{s'}(w)){\,=\,}0,$ i.e. $\tau_{s'}^-(w)=b^-(x_{s'}(w),\xi_{s'}(w)){\,<\,}0$. However, $\tau^-$ is constant and $p^+(w)=0,$ which implies that $\tau_s^-(w)=\tau'>0$ for all $s$.   

Thus, we can re-write (\ref{bichar sys 1}) as
\begin{equation}\label{bichar sys 3}\left\{
\arraycolsep=1pt
\def\arraystretch{1.2}
\begin{array}{rlrl}
\dot{t}_s&=-p^-(\varphi_s), &
\dot{\tau}_s&=0,\\
(\dot{x}_s)_k&=-p^-(\varphi_s) p^+_{\xi_k}(\varphi_s),\qquad &
(\dot{\xi}_s)_k&=p^-(\varphi_s) p_{x_k}^+(\varphi_s),
\end{array}\right.\end{equation} with initial data $(t_s,\tau_s,x_s,\xi_s)\big|_{s=0}{=w}.$
Notice that $t^+{\,=\,}t'+s,$ and so we may re-parameterize (\ref{bichar sys 2}) in terms of $t^+$:
$$ 
\left\{
\arraycolsep=1pt
\def\arraystretch{2.5}
\begin{array}{rlrl}
\dfrac{d}{dt^+}{t}_{t^+-t'}^+&=1 &
\dfrac{d}{dt^+}{\tau}_{t^+-t'}^+&=0,\\
\left(\dfrac{d}{dt^+}{x}_{t^+-t'}^+\right)_k&= p_{\xi_k}^{+}(\varphi_{t^+-t'}^+),\qquad &
\left(\dfrac{d}{dt^+}{\xi}_{t^+-t'}^+\right)_k&=- p_{x_k}^{+}(\varphi_{t^+-t'}^+),
\end{array}
\right.$$ with initial data $\Big(t^+_{t^+-t'},\tau^+_{t^+-t'},  x^+_{t^+-t'},\xi^+_{t^+-t'}\Big)\big|_{t^+=t'}{=w}.$

Next, we re-parameterize (\ref{bichar sys 3}) to change the flow variable from $s$ to $t$ (which can be done since $t_s$ is strictly increasing and hence invertible), generating the system 
\begin{equation*}\left\{
\arraycolsep=1pt
\def\arraystretch{2.5}
\begin{array}{rlrl}
\dfrac{d}{dt}{t}_{s(t)}&=1, &
\dfrac{d}{dt}{\tau}_{s(t)}&=0,\\
\left(\dfrac{d}{dt}{x}_{s(t)}\right)_k&=p^+_{\xi_k}(\varphi_{s(t)}),\qquad &
\left(\dfrac{d}{dt}{\xi}_{s(t)}\right)_k&=-p_{x_k}^+(\varphi_{s(t)}),
\end{array}\right.\end{equation*} which has initial data $\left(t_{s(t)},\tau_{s(t)},x_{s(t)},\xi_{s(t)}\right)\big|_{t=t'}{=w}.$
 An application of uniqueness theory yields that \newline\noindent $\varphi_{s(t)}(w)=\varphi_{t^+-t'}^+(w)$. The converse is similar by reversing the above process. 
\end{proof} 
When working with the factored flow, the decoupling of $(t,\tau)$ and $(x,\xi)$ allows us to project onto the $(x,\xi)$ components of the flow without worrying about loss of information. For this reason, we will write  $\Pi_{x,\xi}\circ \varphi^\pm$ as simply $\varphi^\pm$, where $\Pi_{x,\xi}(t,\tau,x,\xi){\,=\,}(x,\xi)$. Notice that when we project, we are no longer looking at null bicharacteristics but, rather, bicharacteristics with initial data having non-zero $\xi$ component.

Now, we may define all of the corresponding forward and backward trapped and non-trapped sets for the half-wave flows as
\begin{align*}
\Gamma_{tr}^{\pm}&=\left\lbrace w\in T^*\R^3\setminus o:
\sup_{s\geq 0} |x^\pm_{ s}(w)|<\infty\right\rbrace,\\
\Lambda_{tr}^{\pm}&=\left\lbrace w\in T^*\R^3\setminus o:
\sup_{s\geq 0} |x^\pm_{- s}( w)|<\infty\right\rbrace,\\
\Gamma_{\infty}^{\pm}&=\left\lbrace w\in T^*\R^3\setminus o:
|x^\pm_{s}(w)|\rightarrow\infty\ \operatorname{ as }\ s\rightarrow\infty\right\rbrace,\\ 
\Lambda_{\infty}^{\pm}&=\left\lbrace w\in T^*\R^3\setminus o:
|x^\pm_{-s}(w)|\rightarrow\infty\ \operatorname{ as }\ s\rightarrow\infty\right\rbrace.
\end{align*} The trapped and non-trapped sets are
\begin{align*}
    {\Omega}_{tr}^{\pm}=\Gamma_{tr}^\pm\cap \Lambda_{tr}^\pm&\qand  {\Omega}_{\infty}^\pm=\Gamma_{\infty}^\pm\cap \Lambda_{\infty}^\pm,\\
        {\Omega}_{tr}={\Omega}_{tr}^+\cup {\Omega}_{tr}^-
   &\qand
    {\Omega}_{\infty}={\Omega}_{\infty}^+\cup {\Omega}_{\infty}^-.
\end{align*}
Note that the identities
$$   {\Omega}_{tr}=\Pi_{x,\xi}(\Omega^p_{tr})\qand {\Omega}_{\infty}=\Pi_{x,\xi}(\Omega^p_{\infty})$$ hold as an immediate consequence of the factoring. Additionally, the factoring allows us to re-state the geometric control condition as 
$$\left(w\in \Omega_{tr}^+\implies  (\exists s\in\R)\  \left(a(x^+_s(w))>0\right)\right)\qand  \left(w\in \Omega_{tr}^-\implies  (\exists s\in\R)\  \left(a(x^-_s(w))>0\right)\right).$$
 If $w\in\Omega_{tr}$, then it is either trapped with respect the flow generated by $p^+$ or $p^-$ by Proposition \ref{bichar equiv}. If it is trapped with respect to $p^+,$ then there is a time so that $w$ is flowed along a $p^+$-bicharacteristic ray to a place where the damping is positive, and similarly if it is trapped with respect to $p^-$. 
\subsection{Results on the Flow} 
\label{flow results section}
Here, we establish results regarding the trapped/non-trapped sets and scalings for the flows, culminating in an extension of geometric control to bicharacteristic rays bounded either forward or backward in time. These results largely follow the path outlined in \cites{BR14}, although we require certain scaling results in order to utilize homogeneity arguments in later proofs (which were unnecessary in \cites{BR14} due to their use of semiclassical rescaling). In particular, Lemma \ref{nontrap lem} and Propositions \ref{decomp} and \ref{semi GCC} are analogous to results in Chapter 8 of \cites{BR14} (namely, Lemma 8.2 and Propositions 8.3 and 8.4, respectively).

We will start with a scaling result on the flow. 
\begin{prop}\label{scaling}
The flows generated by $p^\pm$ satisfy the scalings
$$\left\{
\begin{aligned}
    x^\pm_{ s}(x,\xi)&=x^\pm_s(x,\lambda \xi),\\
   \lambda \xi^\pm_{ s}(x,\xi)&=\xi^\pm_s(x,\lambda \xi)
\end{aligned}\right.$$
for any $\lambda>0.$ 
\end{prop}
\begin{proof}
Label the functions on the right-hand side as $x_{s,\lambda}^\pm$ and $\xi_{s,\lambda}^\pm$, respectively. Using the homogeneity of $b^\pm$, the left-hand side $(x^\pm_s,\lambda\xi^\pm_s)$ solves the system 

$$\left\{
\arraycolsep=1pt
\def\arraystretch{1.5}
\begin{aligned}
    \dfrac{d}{d s}x_s^\pm&=\nabla_{\xi}p^\pm(x_s^\pm,\xi_s^\pm)=\nabla_\xi p^\pm(x_s^\pm, \lambda \xi_s^\pm), \\
  \dfrac{d}{d s}(\lambda \xi_s^\pm)&=-\lambda \nabla_{x}p^\pm(x_s^\pm,\xi^\pm_s)=-\nabla_x p^\pm(x_s^\pm,\lambda \xi_s^\pm), \\
    (x_s^\pm,\lambda \xi_s^\pm)\big|_{s=0}&=(x,\lambda \xi), 
\end{aligned}\right.$$
while the right-hand side solves
$$\left\{\begin{aligned}
    \frac{d }{d s}x_{s,\lambda}^\pm&= \nabla_{\xi}p^\pm(x_{s,\lambda}^\pm,\xi_{s,\lambda}^\pm ),\\
  \frac{d}{d s}(\xi_{s,\lambda}^\pm )&=- \nabla_{x}p^\pm(x_{s,\lambda}^\pm,\xi_{s,\lambda}^\pm),\\
  (x_{s,\lambda}^\pm,\xi_{s,\lambda}^\pm )\big|_{s=0}&=(x,\lambda \xi).
\end{aligned}\right.$$ 
Applying uniqueness theory completes the proof. 
\end{proof}
 This scaling implies that the trapped/non-trapped sets, and hence geometric control, are entirely determined by unit speed null bicharacteristics, i.e. by what happens on the unit cosphere bundle $S^* \R^3=\{(x,\xi)\in T^*\R^3: |\xi|=1\}.$ Indeed, observe that
 $$ x_s^\pm(x,\xi)=x^\pm_{s}(x,\xi/|\xi|).$$
 The forward/backward trapped sets are defined in terms of supremums of the above over $s$, while the forward/backward non-trapped sets are defined via limits in $s$, and the prior equation shows that all of these are unaffected by the scaling in the $\xi$ component of the initial data. A more pertinent scaling is given by the function $$\Phi^\pm(x,\xi)=\left(x,\frac{\xi}{|b^\pm(x,\xi)|}\right).$$ The utility of this scaling comes from the fact that $b^\pm$ is a constant of motion under the corresponding projected Hamiltonian flows and that  $$\left|{\frac{\xi}{b^+(x,\xi)}}\right|\approx 1,$$ which we now prove. 
 
 \begin{prop}\label{b scale}
 For any $(x,\xi)\in T^*\R^3\setminus o,$ $$\left|{\frac{\xi}{b^\pm(x,\xi)}}\right|\approx 1.$$
 \end{prop}
\begin{proof}
By homogeneity, 
$$\left|{\frac{\xi}{b^\pm(x,\xi)}}\right|=\frac{1}{\left|b^\pm\left(x,\frac{\xi}{|\xi|}\right)\right|}.$$
Write $$\left|b^\pm\left(x,\frac{\xi}{|\xi|}\right)\right|=\left|g^{0j}\frac{\xi_j}{|\xi|}\pm\sqrt{\left(g^{0j}\frac{\xi_j}{|\xi|}\right)^2+(g^{ij}-m^{ij})\frac{\xi_i}{|\xi|}\frac{\xi_j}{|\xi|}+m^{ij}\frac{\xi_i}{|\xi|}\frac{\xi_j}{|\xi|}}\right|.$$ 
Since $\norm{g-m}_{AF(|x|>R_0)}\ll 1$, asymptotic flatness guarantees that $g^{0j}$ and $g^{ij}- m^{ij}$ are small in the exterior region $\{|x|>R_0\}$. Hence,  $$\left|b^\pm\left(x,\frac{\xi}{|\xi|}\right)\right|\approx \sqrt{m^{ij}\frac{\xi_i}{|\xi|}\frac{\xi_j}{|\xi|}}= 1$$ when $|x|>R_0.$ 
 
In the interior region, we are considering $b^\pm$ on the compact set $\{|x|\leq R_0\}\times \{|\xi|=1\}.$ Since we know that $|b^\pm|>0$ for all $\xi\neq 0$ from Proposition \ref{b pm sign}, continuity guarantees the desired boundedness here.
\end{proof} In view of Proposition \ref{b scale}, it follows that $$\frac{\xi_j}{b^\pm(x,\xi)}\in S^0_{\operatorname{hom}}(T^*\R^3\setminus o),\qquad j=1,2,3,$$ where $S^0_{\operatorname{hom}}(T^*\R^3\setminus o)$ denotes the $0^{\operatorname{th}}$-order homogeneous symbol class.
\begin{remark} \label{xi flow bdd} As a consequence of the scaling, the sets  $$
    \dot{\Gamma}_{tr}^\pm=\Gamma_{tr}^\pm\cap\Phi^\pm(T^*\R^3\setminus o)\qand
    \dot{\Lambda}_{tr}^\pm=\Lambda_{tr}^\pm\cap\Phi^\pm(T^*\R^3\setminus o) 
    $$
    are invariant under the flow. Indeed, it is readily seen that the (semi) trapped nature is preserved. Further, since $b^\pm$ is constant along the flow, it follows from Proposition \ref{scaling} that $$\xi_s^\pm\left(x,\frac{\xi}{|b^\pm(x,\xi)|}\right)=\frac{1}{|b^\pm(x,\xi)|}\xi_s^\pm(x,\xi)=\frac{1}{|b^\pm(x_s^\pm(x,\xi),\xi_s^\pm(x,\xi))|}\xi_s^\pm(x,\xi).$$ 
\end{remark} 

Now, we prove a key result on non-trapped trajectories. 
\begin{Lemma} \label{nontrap lem}
If $R\geq R_0$ and $$|x^\pm_{\pm s'}( x,\xi)|\geq \max\{2R, |x|+\delta\}$$ for some $(x,\xi)\in T^*\R^3\setminus o,$ $\delta>0,$ and $s'>0,$ then it holds for all $s\geq s'$, and $$|x^\pm_{\pm s} (x,\xi)|\rightarrow \infty$$ as $s\rightarrow\infty.$ 
\end{Lemma} 
That is, if we can get sufficiently far away from the origin and move radially outward from the initial position, then the trajectories are necessarily non-trapped. This can be proven directly, but the computations are simpler if one uses the correspondence between null bicharacteristics for $p$ and $p^\pm$.
\begin{proof}
Without loss of generality, we will work with the $x^+$ bicharacteristic ray. By Proposition \ref{bichar equiv}, it suffices to prove the result for the null bicharacteristic ray $x_{\pm s}$ with initial data $\tilde{w}$, where $\tilde{w}$ is the lift of $w$ to $T^*\R^4\setminus o$ which is consistent with the comment immediately following the aforementioned proposition (in particular, the $\tau$ component is strictly positive). For any $z\in T^*\R^4\setminus o$, we explicitly calculate that $$\frac{1}{2}\frac{\partial^2}{\partial s^2}|x_{\pm s}(z)|^2=\left|\frac{\partial}{\partial s} x_{\pm s}(z)\right|^2+x_{\pm s}(z)\cdot \frac{\partial^2}{\partial s^2} x_{\pm s}(z),$$ where
\begin{align*}
    \left|\frac{\partial}{\partial s} x_{\pm s}(z)\right|^2=4 \tau_{\pm s}^2(z)\left(\sum\limits_{k=1}^3g^{0k}(x_{\pm s}(z))\right)^2&+4\sum\limits_{k=1}^3 [(g^{ki}(x_{\pm s}(z))(\xi_{\pm s}(z))_i][(g^{kj}(x_{\pm s}(z))(\xi_{\pm s}(z))_j]\\
    &+8\sum\limits_{k=1}^3 \tau_{\pm s}(z) g^{0k}(x_{\pm s}(z))g^{kj}(x_{\pm s}(z))(\xi_j(z))_{\pm s},
\end{align*} and  
\begin{align*}
    &x_{\pm s}(z)\cdot \frac{\partial^2}{\partial s^2} x_{\pm s}(z)\\
    &=4\tau_{\pm s}(z) (x_{\pm s}(z))_k [\partial_\ell g^{0k}(x_{\pm s}(z))]\left(\tau_{\pm s}(z) g^{0\ell}(x_{\pm s}(z))+g^{\ell j}(x_{\pm s}(z))(\xi_{\pm s}(z))_j\right)\\
   &\qquad +4(x_{\pm s}(z))_j[\partial_\ell g^{kj}(x_{\pm s}(z))]\left(\tau_{\pm s}(z) g^{0\ell}(x_{\pm s}(z))+g^{\ell j}(x_{\pm s}(z))(\xi_{\pm s}(z))_j\right)(\xi_{\pm s}(z))_j\\
    &\qquad-2(x_{\pm s}(z))_k g^{kj}(x_{\pm s}(z))\left(2\tau_{\pm s}(z)\partial_j g^{0i}(x_{\pm s}(z))(\xi_{\pm s}(z))_i+\partial_j g^{i\ell}(x_{\pm s}(z))(\xi_{\pm s}(z))_i(\xi_{\pm s}(z))_{\ell}\right).
\end{align*} Since $\tau_s$ is constant for stationary $P$, it follows that 
$$\tau_{\pm s}(z)=\tau_0=b^+ (\Pi_{x,\xi}(z))=b^\pm\left(x^+_{\pm s}(\Pi_{x,\xi}(z)),\xi^+_{\pm s}(\Pi_{x,\xi}(z))\right)\approx  |\xi^+_{\pm s}(\Pi_{x,\xi}(z))|,$$
and so $$\frac{\partial^2}{\partial s^2}\left| x_{\pm s}(z)\right|^2\gtrsim |\xi^+_{\pm s}(\Pi_{x,\xi}(z))|^2\left(1-\norm{g-m}_{AF_{>R}}\right)$$ provided that $|x_{\pm s}(z)|>R$. In such a case, we have that $\norm{g-m}_{AF_{>R}}\ll 1,$ and thus $$\frac{\partial^2}{\partial s^2}|x_{\pm s}(z)|^2>0.$$ Using the condition that $|x^\pm_{\pm s'}(x,\xi)|\geq \max\{2R, |x|+\delta$\} 
in a typical mean value theorem argument (elementary and, thus, omitted) establishes the existence of an $s''\in [0,s']$ such that 

    $$|x_{ \pm s''}(\tilde{w})|^2> R^2 \qand
    \left(\frac{\partial}{\partial s}|x_{\pm s}(\tilde{w})|^2\right)\Big|_{s=s''}>0.$$
All together, we have that $|x_{\pm s}(\tilde{w})|^2$ has positive derivative at $s=s''$, and its derivative is increasing for all $s\geq s''$. In particular, $|x_{\pm s}(\tilde{w})|^2$ is increasing for all $s\geq s'',$ which implies the result.
\end{proof}

As a consequence, we can use the trapped and non-trapped sets to partition phase space.
\begin{prop} \label{decomp} $\ $ 
\begin{enumerate}[label=\normalfont(\alph*)]
    \item  We can partition $T^*\R^3\setminus o$ as \begin{align*} 
        T^*\R^3\setminus o&=\Gamma_{tr}^\pm\sqcup\Gamma_\infty^\pm=\Lambda_{tr}^\pm\sqcup\Lambda_\infty^\pm,\\
        T^*\R^3\setminus o&=\Gamma_{tr}^\pm\cup\Lambda_{tr}^\pm\cup  \Omega_{\infty}^\pm.
    \end{align*} 
    \item 
    $\Gamma^\pm_\infty, \Lambda^\pm_\infty, \Omega^\pm_{\infty}$ are open in $T^*\R^3\setminus o$, and $\Gamma_{tr}^\pm,\Lambda_{tr}^\pm, \Omega^\pm_{tr}$ are closed.
    \item If $K\subset \Omega_\infty^\pm$ is compact, then for every $R\geq R_0$, there exists $T'\geq 0$ so that $$|x^\pm_s (v)|> R$$ for every $|s|\geq T'$ and ${v\in K}.$ Also, $$\bigcup\limits_{s\in\R}\varphi_{s}^\pm (K)$$ is closed in $T^*\R^3\setminus o.$
\end{enumerate}
\end{prop}
We omit the proof of this result, as it follows directly from using Lemma \ref{nontrap lem} and continuity of the flow in both the evolution parameter and the data, as in \cites{BR14} (see Proposition 8.3 in the aforementioned work).

Finally, we show that if one assumes geometric control for bounded bicharacteristic rays, then it holds for semi-bounded bicharacteristic rays (that is, those which are bounded forward or backward in time). Although the proof is similar to that given in \cites{BR14}, we include in here due to its importance in our work and the seemingly increased complexity of our sets (at least notationally).

\begin{prop}\label{semi GCC} Assume that the geometric control condition (\ref{GCC}) holds.  If $w\in \dot{\Gamma}_{tr}^\pm$,  then there exists $s_\pm\geq 0$ so that $a\left(x_{s_\pm}^\pm(w)\right)>0.$ The same is true for $w\in\dot{\Lambda}_{tr}^\pm,$ but with $s_\pm\leq 0.$
\end{prop}
\begin{proof} We will only demonstrate this for $\dot{\Gamma}_{tr}^+$, as the work to establish the remaining cases is similar. If $w\in \dot{\Gamma}_{tr}^+,$ then $$\alpha:=\sup_{s\geq 0}|x^+_s(w)|<\infty.$$ According to Remark \ref{xi flow bdd}, $|\xi_s^+(w)|\approx 1$ for all $s\in\R$. Thus, $$\sup_{s\geq 0} |\varphi_s^+(w)|<\infty.$$

Then, there exists a point $w'\in T^*\R^3$ and a sequence $(s_n)$ of non-negative real numbers such that $\varphi^+_{s_n}(w)\rightarrow w'$ as $s_n\rightarrow \infty.$ For any $s\in\R,$ the group law for the flow tells us that $\varphi^+_{s+s_n}(w)=\varphi_s^+(\varphi^+_{s_n}(w)),$ and so 
$$x^+_{s+s_n}(w)=\Pi_x\circ \varphi_s^+(\varphi^+_{s_n}(w))\rightarrow x_s^+(w') \text{ as }s_n\rightarrow\infty.$$ Since $s+s_n\geq 0$ for large enough $n$, it follows that  $|x_s^+(w')|\leq \alpha$  for all $s\in\R.$ By (\ref{GCC}), there exists $s'\in\R$ for which $ a(x^+_{s'}(w'))>0.$ Recall that  $x_{s'+s_n}^+(w)\rightarrow x^+_{s'}(w')$ as $n\rightarrow\infty$. Since $a$ is continuous and $s'+s_n\geq 0$ for $n$ large enough, we conclude that  $a\left(x_{s'+s_N}^+(w)\right)>0$ for some large $N$.
\end{proof}

\subsection{Escape Function Construction}
\label{escape section}
We will construct our symbols in multiple steps:
\begin{enumerate}
    \item \textbf{On the characteristic set.} Since we are utilizing the half-wave decomposition, working on the characteristic set amounts to working on each individual light cone, then combining together. There are three regions of interest, two sub-regions of the \textit{interior region} $\{|x|\leq R\}$ and the \textit{exterior region} $\{|x|> R\}$. Here, $R\geq R_0$. 
    \begin{enumerate}
        \item \textbf{Interior, semi-bounded null bicharacteristics.} As opposed to working with the trapped and non-trapped sets, we will first work with the semi-bounded null bicharacteristics with initial data living in the interior region $\{|x|\leq R\}$. Working with the trapped and non-trapped sets can be difficult, since one can have non-trapped trajectories which are bounded forward or backward in time (but not both). Heuristically, these trajectories constitute the boundary of the non-trapped set. Instead, we will explicitly work with trajectories which are bounded forward or backward in time. This is where geometric control is used. This step is inspired by the work in \cites{BR14}.
        \item \textbf{The remainder of the interior region.} Since there is no trapping here, we construct a symbol similar to the one constructed in \cites{BR14}, \cites{Doi96}, and \cites{mst20}. We will need to make an appropriate modification to avoid trapped trajectories while working with the half-wave symbols.
        \item \textbf{The exterior region.} As a consequence of asymptotic flatness, there are no trapped trajectories here. Hence, this follows from a similar multiplier to that used to prove local energy decay for the flat wave equation, although the multiplier must be appropriately adapted to the geometry. Here, we are motivated by prior work in \cites{MMT08} and \cites{mst20}.
    \end{enumerate}
    \item \textbf{On the elliptic set.} Here, we construct a correction term. That is, we will construct a lower-order symbol which provides no contribution on the characteristic set and provides positivity off of it. This is based on the work in \cites{mst20}.
\end{enumerate}

We will break this construction up into a sequence of lemmas, starting with (1a). While our construction follows that of \cites{BR14}, we reason differently. Their argument utilizes semiclassical rescaling, which provides compactness for their interior, semi-trapped set. Since we are sticking with the microlocal framework, we instead utilize homogeneity arguments to obtain this compactness. This is one of the reasons to work with the half-wave decomposition (the other being related to step (1b), which we will outline once we get there).

    With this in mind, we will utilize the sets 
\begin{align*}
    \Omega_{R}^\pm &:=\left(\Gamma_{tr}^\pm\cup\Lambda_{tr}^\pm\right) \cap\{|x|\leq R\},\\
    \dot{\Omega}_R^\pm&:=\Omega_R^\pm\cap\Phi^\pm(T^*\R^3\setminus o).
\end{align*}
As a consequence of Proposition \ref{b scale} and Proposition \ref{decomp}(b), the latter set is compact. 

\begin{Lemma}[Semi-bounded Escape Function Construction]\label{semibdd}
There exist $q^{\pm}\in C^\infty(T^*\R^3\setminus o)$, an open set $V_R^{\pm}\supset \Omega_{R}^\pm$, and $C^\pm\in\R_+$ so that $$H_{p^\pm}q^\pm+C^\pm a\gtrsim_R \mathbbm{1}_{V_R^\pm}.$$ Further, $q^\pm=q_1^\pm\circ \Phi^\pm,$ where $q_1^\pm\in C_c^\infty(T^*\R^3\setminus o)$.
\end{Lemma}
Here, $\Phi^\pm\in S^0_{\operatorname{hom}}(T^*\R^3\setminus o)$ is the scaling function introduced in Section \ref{flow results section}.
The fact that we omit the zero section is unavoidable, but it is non-problematic; we will introduce high-frequency cutoffs to our symbols later on which allow for smooth extensions to all of phase space.  
\begin{proof}
 We will first construct a symbol $q_1^\pm$ and an open set $\dot{V}_R^\pm\supset \dot{\Omega}_R^\pm$ such that $$H_{p^\pm}q_1^\pm+C^\pm a\gtrsim_R \mathbbm{1}_{\dot{V}_R^\pm}.$$ To that end, let $w^\pm\in \dot{\Omega}_R^\pm.$ By Proposition \ref{semi GCC}, there exists $s_{w^\pm}\in\R$ for which $a(x^\pm_{s_{w^\pm}}(w^\pm))>0$. Say that $2\alpha_{w^\pm}:=a(x^\pm_{s_{w^\pm}}(w^\pm)).$ By the continuity of the flow in the initial data, there exists a neighborhood $U_{w^\pm}$ of $w^\pm$ so that $a(x^\pm_{s_{w^\pm}}(z))>\alpha_{w^\pm}$ for all $z\in U_{w^\pm}$. Select a smooth cutoff $\chi_{w^\pm}\in C_c^\infty(T^*\R^3)$ so that $\supp\ \chi_{w^\pm}\subset U_{w^\pm}$ and $\chi_{w^\pm}\equiv 1$ on a smaller neighborhood $V_{w^\pm}$ of ${w^\pm}$. Now, we define a symbol on $T^*\R^3\setminus o$ given by 
$$q_{w^\pm}(x,\xi)=\int\limits_0^{s_{w^\pm}}\left(\chi_{w^\pm}\circ \varphi^\pm_{-s}\right)(x,\xi)\, ds.$$

Such a symbol is readily seen to be well-defined, and it is smooth by the aforementioned smooth flow dependence on data. Next, we demonstrate its symbolic nature. By continuity of the flow,\newline\noindent $\varphi^\pm_{[0,s_{w^\pm}]}(\overline{U_{w^\pm}}):=\varphi^\pm([0,s_{w^\pm}]\times \overline{U_{w^\pm}})$ is compact. 
If $(x,\xi)\notin \varphi^\pm_{[0,s_{w^\pm}]}(\overline{U_{w^\pm}}),$ then $(x,\xi)\notin \varphi^\pm_s(\overline{U_{w^\pm}})$ for any \newline\noindent $s\in [0,s_w]$. Then, $\varphi_{-s}^\pm(x,\xi)\notin \overline{U_{w^\pm}}$ for any $s\in [0,s_{w^\pm}]$, implying that $q_{w^\pm}(x,\xi)=0.$ Hence, \newline\noindent $q_{w^\pm}\in C_c^\infty(T^*\R^3\setminus o).$ %\newline\noindent to avoid inline equation breaks

Applying the Hamiltonian vector field $H_{p^\pm}$ gives us
$$H_{p^\pm}q_{w^\pm}=\int\limits_0^{s_{w^\pm}} H_{p^\pm} (\chi_{w^\pm}\circ \varphi_{-s}^\pm)\, ds=-\int\limits_0^{s_{w^\pm}} \partial_s\left( \chi_{w^\pm}\circ \varphi_{-s}^\pm\right)\, ds=\chi_{w^\pm}-\chi_{w^\pm}\circ\varphi_{-s_{w^\pm}}^\pm.$$ Notice that the term $-\chi_{w^\pm}\circ\varphi_{-s_{w^\pm}}^\pm$ is non-positive and that  $$\supp\left(\chi_{w^\pm}\circ \varphi_{-s_{w^\pm}}^\pm\right)\subset \left\lbrace v: \varphi_{-s_{w^\pm}}^\pm(v)\in U_{w^\pm}\right\rbrace=\left\lbrace v: v\in \varphi_{s_{w^\pm}}^\pm(U_{w^\pm})\right\rbrace\subset \{x: a(x)>\alpha_{w^\pm}\}.$$ 
Using this support property, we can use the damping to absorb the poorly-signed term and obtain non-negativity of $H_{p^\pm}q_{w^\pm}$. Indeed, if we call $C_{w^\pm}=2(\alpha_{w^\pm})^{-1}$, then we have $$\chi_{w^\pm}\circ\varphi_{-s_{w^\pm}}+C_{w^\pm}a(x)\geq 0.$$  Thus, 
$$H_{p^{\pm}}q_{w^\pm}+C_{w^\pm} a\gtrsim\mathbbm{1}_{V_{w^\pm}}.$$ 

Since $\dot{\Omega}_R^\pm$ is compact, we can reduce the open cover $\{V_{w^\pm}\}_{{w^\pm}\in \dot{\Omega}_R^\pm}$ to a finite subcover  $\{V_{w_j^\pm}\}_{j=1}^m,$ with each $w_j^\pm\in \dot{\Omega}_R^\pm.$ Call $$\dot{V}_R^\pm=\bigcup\limits_{j=1}^m V_{w^\pm_j},\qquad q^\pm_1=\sum\limits_{j=1}^m q_{w_j^\pm},\qquad \text{and}\qquad {C}^\pm=\sum\limits_{j=1}^mC_{w_j^\pm}.$$
 This provides us with a symbol $q_1^\pm\in C_c^\infty(T^*\R^3\setminus o)$ so that $$H_{p^\pm}q_1^\pm + {C}^\pm a\gtrsim\mathbbm{1}_{\dot{V}_R^\pm},\qquad \dot{V}_R^\pm\supset \dot{\Omega}_R^\pm.$$ 
 
 Finally, we will extend the above estimate from an indicator on $\dot{V}_R^\pm$ to an indicator on a neighborhood $V_R^\pm\supset\Omega_R^\pm$. Consider the function $q^{\pm}:T^*\R^3\setminus o\rightarrow\R$ given by $$q^\pm=q_1^\pm\circ\Phi^\pm.$$ Since geometric control is invariant under $\Phi^\pm,$ we can see that $q^\pm\neq 0.$ 
 By definition, $$H_{p^\pm} q^\pm\big|_{(x,\xi)}=\frac{d}{ds}\left(q^\pm(x_s^\pm,\xi_s^\pm)\right)\big|_{s=0}.$$ Since $b^\pm$ is a constant of motion for the Hamiltonian system generated by $p^\pm$, it follows that $$(\nabla_x b^\pm)(x_s^\pm,\xi_s^\pm)\dot{x}_s^\pm+(\nabla_{\xi} b^\pm)(x_s^\pm,\xi_s^\pm)\dot{\xi}_s^\pm=0$$ for all $s$. Using this, we calculate that 
 \begin{align*}
     \frac{d}{ds}\left(q^\pm(x_s^\pm,\xi_s^\pm)\right) &= \frac{d}{ds}\left(q_1^\pm\left(x_s^\pm,\frac{\xi_s^\pm}{|b^\pm(x_s^\pm,\xi_s^\pm)|}\right)\right)\\
     &=(\nabla_x q_1^\pm)\left(x_s^\pm,\frac{\xi_s^\pm}{|b^\pm(x_s^\pm,\xi_s^\pm)|}\right)\cdot \left(\dot{x}_s^\pm\right)\\
     &\hspace*{-1in}+(\nabla_{\xi}q_1^\pm)\left(x_s^\pm,\frac{\xi_s^\pm}{|b^\pm(x_s^\pm,\xi_s^\pm)|}\right)\cdot \frac{(|b^\pm(x_s^\pm,\xi_s^\pm)|\dot{\xi}_s^\pm-\xi_s^\pm\left((\nabla_x b^\pm)(x_s^\pm,\xi_s^\pm)\dot{x}_s^\pm+(\nabla_{\xi} b^\pm)(x_s^\pm,\xi_s^\pm)\dot{\xi}_s^\pm\right)}{|b^\pm(x_s^\pm,\xi_s^\pm)|^2}\displaybreak\\ 
     &=(\nabla_x q_1^\pm)\left(x_s^\pm,\frac{\xi_s^\pm}{|b^\pm(x_s^\pm,\xi_s^\pm)|}\right)\cdot (\nabla_{\xi} p^\pm)(x_s^\pm,\xi_s^\pm)\\
    & \qquad-\frac{1}{|b^\pm(x_s^\pm,\xi_s^\pm)|}(\nabla_{\xi}q_1^\pm)\left(x_s^\pm,\frac{\xi_s^\pm}{|b^\pm(x_s^\pm,\xi_s^\pm)|}\right)\cdot (\nabla_x p^\pm)(x_s^\pm,\xi_s^\pm)\\
      &=(\nabla_x q_1^\pm)\left(x_s^\pm,\frac{\xi_s^\pm}{|b^\pm(x_s^\pm,\xi_s^\pm)|}\right)\cdot (\nabla_{\xi} p^\pm)\left(x_s^\pm,\frac{\xi_s^\pm}{|b^\pm(x_s^\pm,\xi_s^\pm)|}\right)\\
     &\qquad-(\nabla_{\xi}q_1^\pm)\left(x_s^\pm,\frac{\xi_s^\pm}{|b^\pm(x_s^\pm,\xi_s^\pm)|}\right)\cdot (\nabla_x p^\pm)\left(x_s^\pm,\frac{\xi_s^\pm}{|b^\pm(x_s^\pm,\xi_s^\pm)|}\right)\\
     &=H_{p^\pm} q_1^\pm\big|_{\left(x_s^\pm,\frac{\xi_s^\pm}{|b^\pm(x_s^\pm,\xi_s^\pm)|}\right)},
 \end{align*}
 where we have used homogeneity to obtain that  $$ (\nabla_\xi p^\pm)\left(x_s^\pm,\xi_s^\pm\right)=(\nabla_\xi p^\pm)\left(x_s^\pm,\frac{\xi_s^\pm}{|b^\pm(x_s^\pm,\xi_s^\pm)|}\right)$$ and $$ \frac{1}{|b^\pm(x_s^\pm,\xi_s^\pm)|}(\nabla_x p^\pm)\left(x_s^\pm,\xi_s^\pm\right)=(\nabla_x p^\pm)\left(x_s^\pm,\frac{\xi_s^\pm}{|b^\pm(x_s^\pm,\xi_s^\pm)|}\right).$$ If we define $V_R^\pm=(\Phi^\pm)^{-1}\left(\dot{V}_R^\pm\right),$ then we have an open neighborhood of $\Omega_R^\pm$ such that $$H_{p^\pm} q^\pm\big|_{(x,\xi)}+C^\pm a(x)=H_{p^\pm} q_1^\pm \big|_{\left(x,\frac{\xi}{|b^\pm(x,\xi)|}\right)}+C^\pm a(x)\gtrsim\left(\mathbbm{1}_{\dot{V}_R^\pm}\circ \Phi^\pm \right)(x,\xi)\geq \mathbbm{1}_{V_R^\pm},$$ since $\Phi^\pm(V_R^\pm)\subset \dot{V}_R^\pm$. 
\end{proof}
Now that we have completed step (1a), we move on to parts (1b) and (1c). Step (1b) pertains to non-trapped null bicharacteristics in the interior region. The symbol that we produce follows the construction appearing in \cites{Doi96} and utilized in many other works, such as \cites{BR14} and \cites{mst20}. Like in \cites{mst20}, we perform a factoring argument. The reason for studying the half-wave decomposition is due to the presence of a cutoff needed to make our constructed ``symbol'' genuinely symbolic. In the unfactored setting, cross terms in the metric arise when differentiating the cutoff in the computation of the Poisson bracket, generating an error term that is difficult to control. In the factored setting, this error can be handled straightforwardly. 

Step (1c) takes place in the exterior region. This is of little concern, as we possess robust exterior estimates. 
We utilize this symbol as a means of bootstrapping the aforementioned error term, which will be compactly-supported in the region where the exterior symbol has strictly positive Poisson bracket with $p^\pm$. 

To these ends, we will analyze both half-waves simultaneously (as in Lemma \ref{semibdd}). 
While this portion of the argument follows the one given in \cites{mst20}, it does require a modification; the escape function on interior, non-trapped null bicharacteristics needs an appropriate adjustment to ensure that it avoids trapped trajectories. We start with a proposition where we construct a function that will be used for the previously-described error absorption. The construction of this function comes from e.g. \cites{mst20}, \cites{Tat08}.

\begin{prop} \label{bootstrap fcn} Let $\sigma>0.$ Then, there exists $f\in C^\infty$ satisfying $f(r)\approx_\sigma 1$ when $r> R_0$ and $f'(r)\approx \sigma c_j 2^{-j}f(r)$ when $r\approx 2^j>R_0$.
\end{prop}
Here, $(c_j)$ is the slow-varying sequence introduced in Section \ref{setup}.
\begin{remark}
Although the sequence $(c_j)$ is not defined for all natural numbers, the indices where it is not defined index finitely many dyadic regions (in particular, they omit where the operator $P$ need not be a small $AF$ perturbation). Since this region is compact, we can extend the sequence to such indices in an arbitrary manner. The typical way that this sequence is extended is by choosing $c_j$ so that $\norm{g-m}_{AF(A_j)}\lesssim c_j$ for the previously-undefined indices $j$.
\end{remark}
\begin{proof}
    As in \cites{Tat08}, we can construct a smooth function $c(s)$ from the sequence $(c_j)$ such that $c(s)\in (c_j, 2c_j)$ for each $s\in (2^j, 2^{j+1})$ and $|c'(s)|\leq\delta s^{-1}c(s).$ Since $(c_j)$ is a positive sequence which converges to zero, it has a positive maximum, say $c_N$. Then, we observe that
    $$\textbf{c}\lesssim c_N\leq c(2^N+2^{N-1})=\left|\int_{2^N+2^{N-1}}^\infty c'(s)\, ds\right|\lesssim\int\limits_1^\infty \frac{c(s)}{s}\, ds$$
    and $$\int\limits_1^\infty\frac{c(s)}{s}\, ds\leq \sum_{j=0}^\infty \int\limits_{2^j}^{2^{j+1}} \frac{2c_j}{2^j}\, ds=2\sum_{j=0}^\infty c_j\lesssim \textbf{c}.$$ That is, $$\int\limits_1^\infty\frac{c(s)}{s}\, ds\approx\textbf{c}.$$  Now, set 
    $$f(r)=\exp\left(\sigma\int\limits_1^{r}\frac{c(s)}{s}\, ds\right).$$ From our prior estimate, it is immediate that $$f(r) \approx e^{\sigma \textbf{c}}\approx_\sigma 1$$ for $r>R_0,$ and $$ f'(r)= \sigma\frac{ c(r)}{r} f(r)\approx \sigma c_j2^{-j}f(r)$$ for $r\approx 2^j. $
\end{proof}
Now, we complete steps (1b) and (1c).
\begin{Lemma}[Non-trapped Escape Function Construction]\label{non-trap}
Let $R\geq R_0$. Then, there exist\newline\noindent $q^\pm\in C^\infty(T^*\R^3\setminus o)$ and $W^\pm\subset \Omega_\infty^\pm$ 
so that $V_R^\pm\cup W^\pm=T^*\R^3\setminus o$ and $$H_{p^\pm}q^\pm
\gtrsim c_j2^{-j}\mathbbm{1}_{W^\pm}, \qquad |x|\approx 2^{j}.$$ Further, $q^\pm=\varepsilon q_{in}^\pm+q_{out}^\pm,$ where $q_{in}^\pm=\tilde{q}_{in}^\pm\circ \Phi^\pm$ with $\tilde{q}_{in}^\pm\in C^\infty(T^*\R^3\setminus o)$ is supported in $\{|x|\leq 4R\}$, $q_{out}^\pm\in S^0_{\operatorname{hom}}(T^*\R^3\setminus o)$, and $\varepsilon>0$ is sufficiently small.   
\end{Lemma}

The inclusion of the sequence $(c_j)$ is necessitated by the prior proposition, which is used for bootstrapping purposes in the exterior region. Its slowly varying nature allows one to work in the weight $\inprod{x}^{-2}$ from the powers $|x|\approx 2^{-j}$ which will arise in the exterior (there is no trouble working in the weight $\inprod{x}^{-2}$ in the interior region by compactness).
\begin{proof}
 Recall from Proposition \ref{b scale} that $|\xi|\approx |b^\pm(x,\xi)|$ on $T^*\R^3\setminus o;$ let $c^\pm, C^\pm > 0$ denote the respective lower and upper bound implicit constants in the inequalities and take $\delta^\pm$ such that $c^\pm-\delta^\pm>0.$ Now, choose $\psi^\pm\in C_c^\infty(T^*\R^3\setminus o)$ such that  \begin{align*}
\supp\psi^\pm&\subset  \Omega_\infty^\pm \cap\{|x|\leq R\}\cap \{c^\pm-\delta^\pm<|\xi|<C^\pm+1\},\\
\psi^\pm\equiv 1\text{ on }U_R^\pm&:=\left(\Omega_\infty^\pm\cap\{|x|\leq R\} \cap \Phi^\pm(T^*\R^3\setminus o)\right)\setminus \dot{V}_R,\end{align*}  where $R\geq R_0.$ 
Now, we define the function $$\tilde{q}_{in}^\pm(x,\xi)=-\chi_{<2R}(|x|)\int\limits_0^\infty \psi^\pm\circ \varphi_s^\pm(x,\xi)\, ds,\qquad (x,\xi)\in T^*\R^3\setminus o.$$ Since non-trapped null bicharacteristic rays must exit any compact set after a finite amount of time, this integral is well-defined for each $(x,\xi)\in T^*\R^3\setminus o,$ which establishes $\tilde{q}_{in}^\pm$ as a well-defined function. It takes more work to show that  $\tilde{q}_{in}^\pm$ is smooth. Similar to \cites{BR14}, we will begin by establishing a maximal amount of time that bicharacteristic rays can remain in the support of the integrand. We already know that $\supp\psi^\pm$ is compact. Let $V^\pm$ be an open neighborhood of $\supp\psi^\pm$ such that $\overline{V^\pm}\subset \Omega_\infty^\pm$. Take $\overline{V^\pm}=K$ in Proposition \ref{decomp}(c), and let $T'$ be as given in the proposition. 
We claim that every point  $w^\pm\in T^*\R^3\setminus o$ has a neighborhood $U_{w^\pm}$ of $w^\pm$ and a time $s_{w^\pm}\geq 0$ such that $ (\psi^\pm\circ\varphi_s^\pm)(z)=0 $ for every $z\in U_{w^\pm}$ and $s\in\R_+\setminus [s_{w^\pm},s_{w^\pm}+T'].$ That is, all  bicharacteristics (with speed $\approx 1$) can spend no more than time $T'$ within $\supp\psi^\pm.$ The time $s_{w^\pm}$ bears no similarity to the variable of the same name in the proof of Proposition \ref{semibdd}.
As a direct consequence of Proposition \ref{decomp}, we may take $U_{w^\pm}=V^\pm$ and $s_{w^\pm}=0$ whenever $w^\pm\in\supp\psi^\pm\subset V^\pm.$ If  $$w^\pm\notin\bigcup\limits_{s\in\R}\varphi_{-s}^\pm\left(\supp\psi^\pm\right)=:\mathcal{X}^\pm,$$ then the fact that $\mathcal{X}^\pm$ is closed provides an open neighborhood  $U_{w^\pm}$ of $w^\pm$ such that $\mathcal{X}^\pm\cap U_{w^\pm}=\emptyset.$ For each $z\in U_{w^\pm}$, we have that $\varphi^\pm_s (z)\notin \supp\psi^\pm$ for all $s\in\R$, i.e. $(\psi^\pm\circ\varphi_s^\pm)(z)=0$ for $s\in \R$. Hence, this case holds with $U_{w^\pm}$ as defined and $s_{w^\pm}=0$. Finally, let  $w^\pm\in \mathcal{X}^\pm\setminus \supp\psi^\pm.$  Then,  $\varphi_{s'}^\pm(w^\pm)\in \supp\psi^\pm$ for some $s'\in\R\setminus \{0\}.$ If $s'>0,$ then we can combine this with the fact that $\varphi_0^\pm(w)\notin\supp\psi^\pm$ and the continuity of the flow to obtain $s_{w^\pm}>0$ such that $\varphi_{s_{w^\pm}}^\pm(w)\in V^\pm$ and $\varphi_s^\pm(w)\notin\supp\psi^\pm$ for all $s\in [0,s_{w^\pm}]$. By continuity of the flow in the data, we can extend the above to a neighborhood $U_{w^\pm}$. That is, there exists a neighborhood  $U_{w^\pm}$ of $w^\pm$ so that for all $z\in U_{w^\pm}$, we have that $\varphi^\pm_{s_{w^\pm}}(z)\in V^\pm$ and $(\psi^\pm\circ \varphi_s^\pm)(z)=0$ for all $s\in [0,s_{w^\pm}].$
Applying Proposition \ref{decomp} to $K=V^\pm$ implies that $(\psi^\pm\circ \varphi_s^\pm)(z)=0$ for all $z\in U_{w^\pm}$ and $s\in [0,s_{w^\pm}]\cup [s_{w^\pm}+T',\infty).$ It remains to consider if we cannot assume that $s'>0$. In this case,  
$$w^\pm\notin\bigcup\limits_{s\in\R_+}\varphi_{-s}^\pm(\supp\psi^\pm)=:\mathcal{X}^\pm_{-}$$ Note that $\mathcal{X}^\pm_{-}$ is closed by the same logic which showed that $\mathcal{X}^\pm$ is closed (see the proof in Proposition \ref{decomp}(c)). From here, one can simply proceed as in the case where $w^\pm\notin\mathcal{X}^\pm.$ 

Using this result, we know that the integral present in $\tilde{q}_{in}^\pm$ is always over an interval of maximal length $T'$. Hence, differentiation under the integral sign is non-problematic and in view of the regularity of the flow map, we conclude that $\tilde{q}_{in}^\pm\in C^\infty(T^*\R^3\setminus o)$. Additionally, it is supported in $\{|x|\leq 4R\}$. In particular, it is smooth and bounded in all derivatives on the compact set $$\{|x|\leq 4R\}\cap \Phi^\pm(T^*\R^3\setminus o).$$ 

Now, consider the smooth function $$q_{in}^\pm=\tilde{q}_{in}^\pm\circ\Phi^\pm$$ defined on $T^*\R^3\setminus o.$  
 As in the proof of Lemma \ref{semibdd}, we get that $$H_{p^\pm}q_{in}^\pm\big|_{(x,\xi)}=H_{p^\pm}\tilde{q}_{in}^\pm\big|_{\Phi^\pm(x,\xi)}.$$ Now, we calculate that
\begin{align*}H_{p^\pm} \tilde{q}_{in}^\pm\big|_{\Phi^\pm (x,\xi)}&=\chi_{<2R}(|x|)\psi^\pm\left(x,\frac{\xi}{|b^\pm(x,\xi)|}\right)\\
&\qquad+\frac{1}{2R}b_{\xi_k}^\pm\left(x,\frac{\xi}{|b^\pm(x,\xi)|}\right)\frac{x_k}{|x|}\chi'\left(\frac{|x|}{2R}\right)\int\limits_0^\infty \psi^\pm\circ \varphi^\pm_s\left(x,\frac{\xi}{|b^\pm(x,\xi)|}\right)\, ds.
\end{align*}
 The first term is non-negative, supported in  $\Omega_\infty^\pm \cap \{|x|\leq R\}$, and equal to $1$ on $U^\pm:=\Phi^{-1}(\dot{U}_R^\pm)$. The second term is an error term which is supported in $\{2R\leq |x|\leq 4R\}.$  The primary purpose of the exterior multiplier is to absorb this error term. To that end, let $$q_{out}^\pm =-\chi_{>R}(|x|)f(|x|) b_{\xi_k}^\pm \frac{x_k}{|x|},$$ where $f$ is the function constructed in Proposition \ref{bootstrap fcn}. It is easy to see that $q_{out}^\pm\in S^0_{\operatorname{hom}}(T^*\R^3\setminus o)$, as it is smooth, bounded in all $x$ derivatives due to asymptotic flatness, homogeneous of degree $0$, and satisfies the appropriate symbol estimate. One can readily compute that 
\begin{align*}
    H_{p^\pm} q^\pm_{out}
    &=b^\pm_{\xi_k}\frac{x_k}{|x|} \chi_{>R}(|x|)f'(|x|)b^\pm_{\xi_j}\frac{x_j}{|x|}
    +b_{\xi_k}^\pm \left(\delta_{jk}-\frac{x_jx_k}{|x|^2}\right)\chi_{>R}(|x|)\frac{f(|x|)}{|x|}\left(\delta_{jl}-\frac{x_jx_l}{|x|^2}\right)b_{\xi_l}^\pm\\
    &\qquad+R^{-1}\chi'\left(\frac{|x|}{R}\right)b^\pm_{\xi_k}\frac{x_k}{|x|}f(|x|)b^\pm_{\xi_j}\frac{x_j}{|x|}
    +\mathcal{O}(\inprod{x}|\partial g|)\chi_{>R}(|x|)|x|^{-1}.
\end{align*}

We remark that the last term is small for $|x|>R$ by asymptotic flatness (and it is localized to this region due to the cutoff), while the remaining terms are all non-negative. The third term is non-negative and supported in the annulus $\{R\leq |x|\leq 2R\}$ due to the support of $\chi'$. Making $\sigma$ large enough and using asymptotic flatness provides that, for any $|x|\approx 2^j$,
\begin{align*}
    H_{p^\pm}q_{out}^\pm &>\frac{\sigma}{2}c_j2^{-j}f(|x|)\chi_{>R}(|x|)\frac{|x\cdot \nabla_{\xi}b^\pm|^2}{|x|^2}+\chi_{>R}(|x|)\frac{f(|x|)}{|x|}\left(|\nabla_\xi b^\pm|^2-\frac{|x\cdot\nabla_\xi b^\pm|^2}{|x|^2} \right)\\
    &\gtrsim c_j2^{-j}
    \chi_{>R}(|x|)|\nabla_\xi b^{\pm}|^2\\
    &\gtrsim c_j2^{-j}\chi_{>R}(|x|).
\end{align*} Thus, $H_{p^\pm}q_{out}^\pm$ is non-negative, strictly positive for $|x|>R$, and  $$H_{p^\pm}q_{out}^\pm \gtrsim c_j2^{-j}\chi_{>R},\qquad |x|\approx 2^{j}.$$ Recall that the error term in $H_{p^\pm}q_{in}^\pm$ is bounded and supported in $\{2R\leq|x|\leq 4R\}$, and $H_{p^\pm}q_{out}^\pm$ is strictly positive on the support of this error (with a uniform bound from below on this set).

Define
$$q^\pm=\varepsilon q_{in}^\pm+q_{out}^\pm\in C^\infty(T^*\R^3\setminus o),$$ where $0<\varepsilon\ll 1$. By choosing $\varepsilon$ sufficiently small, we may absorb the aforementioned error due to our prior discussion, obtaining that
$H_{p^\pm}q^\pm$  is non-negative everywhere and positive on $$W^\pm:=U^\pm\cup\{(x,\xi)\in T^*\R^3\setminus o: |x|>R\}.$$ By Proposition \ref{decomp}(a), \begin{align*}
V_R^\pm\cup U^\pm&=V_R^\pm\cup\left((\Omega_\infty^\pm\cap \{|x|\leq R\})\setminus V_R^\pm \right)\\
&\supset \left(\Omega_R^\pm\cup \Omega_\infty\right)\cap\{|x|\leq R\}\\
&=\left(T^*\R^3\setminus o\right)\cap\{|x|\leq R\},
\end{align*} and so 
\begin{align*}
V_R^\pm\cup U^\pm&=\left(T^*\R^3\setminus o\right)\cap\{|x|\leq R\},\\
V_R^\pm\cup W^\pm&=T^*\R^3\setminus o.
\end{align*}
We have already shown that $$H_{p^\pm}q^\pm\approx 1\qquad (x,\xi)\in U^\pm$$ and $$H_{p^\pm}q^\pm\gtrsim c_j2^{-j} \chi_{>R},\qquad |x|\approx 2^j.$$ The latter estimate readily extends to
$$H_{p^\pm}q^\pm\gtrsim c_j2^{-j} \mathbbm{1}_{W^\pm},\qquad |x|\approx 2^j$$ by the compactness of the interior region $\{|x|\leq R\}$.
\end{proof}
Now, we combine on the light cones to get our desired symbol $q$, as well as obtain positivity on the elliptic set (step (2)). This largely follows the steps present in \cites{mst20}, although we have additional technicalities resulting from the damping.

\begin{proof}[(Proof of Lemma \ref{escape})]
Let $q_1^\pm$ denote the symbol $q^\pm$ constructed in Lemma \ref{semibdd} (not the symbol $q_1^\pm$ from the same lemma) and $q_2^\pm$ denote the symbol $q^\pm$ constructed in Lemma \ref{non-trap}. We remark that, as a consequence of the chain rule, both symbols satisfy the standard $S^0$ bounds for $|\xi|\geq 1.$  First, we truncate to the high-frequency regime via the symbols 
 $$q^\pm_{j,>\lambda}=e^{-\sigma q_j^\pm}\chi_{>\lambda}( |b^\pm|),\ j=1,2,$$
  where $\sigma$ is the parameter in Proposition \ref{bootstrap fcn}. 
We assume that $\lambda>1.$ The exponentiation is implemented for bootstrapping: taking derivatives of the exponential will provide multiplication by $\sigma\gg 1$. Since $|b^\pm| \approx |\xi|,$ these cutoffs genuinely truncate to high frequencies when $\lambda$ is large. Further, the truncation to $|\xi|\gtrsim 1$ eliminates the singularities of $q^\pm_j,$ i.e. $q_j^\pm\chi_{>\lambda}( |b^\pm|)$ smoothly extends to an element of $S^0(T^*\R^3).$ 

We claim that exponentiation preserves the symbol class, so that $q_{j,>\lambda}^\pm\in S^0(T^*\R^3).$ We can immediately see that $q_{j,>\lambda}^\pm$ is smooth. Note that for $|\xi|\geq \lambda,$ the exponentials $e^{-\sigma q_j^\pm}$ are bounded since $q_j^\pm$ are bounded, and for $|\xi|<\lambda$, we immediately have that $q_{j,>\lambda}^\pm\equiv 0$. When checking the symbolic nature of $q_{j,>\lambda}^\pm$, we only need to study the boundedness of the $\xi$ derivatives since our symbols $q_j^\pm$ are bounded in all derivatives in $x$. Taking a partial derivative in $\xi$ provides that 
$$\partial_{\xi_k}q^\pm_{j,>\lambda}=-\sigma(\partial_{\xi_k}q_j^\pm)q^\pm_{j,>\lambda}\mp\frac{e^{-\sigma q^\pm_j}}{\lambda}(\partial_{\xi_k} b^\pm)\chi'\left(\frac{|b^\pm|}{\lambda}\right).$$ 
The first term is $\mathcal{O}(\inprod{\xi}^{-1})$, and the second term is compactly-supported in $\xi$.  
 Due to the aforementioned compact support, we only need consider further $\xi$ differentiation of $\sigma (\partial_{\xi_k}q_j^\pm)q^\pm_{j,>\lambda}$. If the $\xi$ derivative lands on the exponential, then the result is $\mathcal{O}(\inprod{\xi}^{-2})$ by the prior argument. If the derivative lands on $\partial_\xi  q_j^\pm$, then the same asymptotics hold since $\partial_\xi  q_j^\pm\in S^{-1}(T^*\R^3).$ If the derivative lands on the cutoff, then the result is compactly-supported in $\xi$. Inducting  establishes that $q^\pm_{j,>\lambda}\in S^0(T^*\R^3)$.

Now, we combine the symbols constructed on each light cones together as $$q(\tau,x,\xi)=(\tau-b^+)(q^-_{1,>\lambda}+q^-_{2,>\lambda})+(\tau-b^-)(q^+_{1,>\lambda}+q^+_{2,>\lambda}).$$ 
Calling  $$q_j=(\tau-b^+)q^-_{j,>\lambda}+(\tau-b^-)q^+_{j,>\lambda},$$ we can see that 
\begin{align*}
    (H_p q+2\gamma\tau aq)\big|_{\tau=b^\pm}&=(H_p q_1+2\gamma\tau aq_1)\big|_{\tau=b^\pm}+(H_p q_2+2\gamma\tau aq_2)\big|_{\tau=b^\pm}\\
    &=H_p q_1 \big|_{\tau=b^\pm}\pm 2\gamma b^\pm (b^+-b^-)aq_{1,>\lambda}^\pm\\
    &\qquad+H_p q_2 \big|_{\tau=b^\pm}\pm 2\gamma b^\pm (b^+-b^-)aq_{2,>\lambda}^\pm.
\end{align*} We will work with each term in the last equality separately. First, we compute that \begin{align*}
    H_p q_j\big|_{\tau=b^\pm}&=-(b^+-b^-)^2H_{p^\pm} q_{j,>\lambda}^\pm\\
    &\qquad-(b^\pm-b^\mp)q_{j,>\lambda}^\pm(b_{\xi_j}^\pm b_{x_j}^\mp-b^\pm_{x_j}b^\mp_{\xi_j})\\
    &= \sigma (b^+-b^-)^2 q_{j,>\lambda}^\pm H_{p^\pm} q_j^\pm\\
    &\qquad-(b^\pm-b^\mp)q_{j,>\lambda}^\pm(b_{\xi_j}^\pm b_{x_j}^\mp-b^\pm_{x_j}b^\mp_{\xi_j}).
\end{align*}  By making $\sigma$ sufficiently large, we get that 
\begin{align*}
   H_p q_j\big|_{\tau=b^\pm}&\geq \frac{1}{2}\sigma(b^+-b^-)^2q_{j,>\lambda}^\pm H_{p^\pm} q_j^\pm+E_j^\pm, 
\end{align*} where $E_j^\pm$ are error terms which are supported in a neighborhood of the region where $H_{p^\pm} q_j^\pm=0$. These terms are non-problematic, as they are readily absorbed into the above estimate with differing $j$ when we combine the estimates together. Hence, we will drop the $E_j^\pm$'s for ease of notation. 

Observe that $$\frac{b^\pm}{b^\pm- b^\mp}\approx 1.$$ By choosing $\gamma$ large enough, we may apply Lemma \ref{semibdd} to obtain that

\begin{align}\label{highfreqp1}
(H_p q_1+2\gamma \tau aq_1)\big|_{\tau=b^\pm}&\geq \frac{1}{2}\sigma(b^+-b^-)^2q_{1,>\lambda}^\pm H_{p^\pm} q_1^\pm\pm 2\gamma b^\pm (b^+-b^-)aq_{1,>\lambda}^\pm \\
\nonumber
&=\frac{1}{2}\sigma(b^+-b^-)^2q_{1,>\lambda}^\pm\left(H_{p^\pm} q_1^\pm+\left(\frac{4\gamma}{\sigma}\right)\frac{b^\pm}{b^\pm-b^\mp}a\right)\\
\nonumber
&\gtrsim |\xi|^2q_{1,>\lambda}^\pm\left(H_{p^\pm} q_1^\pm+\frac{\gamma}{\sigma}a\right)\\
\nonumber
&\gtrsim \mathbbm{1}_{|\xi|\geq\lambda}\mathbbm{1}_{{V_R^\pm}}|\xi|^{2}. 
\end{align}
Notice that $\gamma$ depends on $\textbf{c}$ (and $\sigma)$.
For the $j=2$ term, we use the prior computation, the fact that the damping term has positive sign, and Lemma \ref{non-trap}:

\begin{align}\label{highfreqp2}
    (H_p q_2+2\gamma\tau aq_2)\big|_{\tau=b^\pm}&\geq \frac{1}{2}\sigma(b^+-b^-)^2q_{2,>\lambda}^\pm H_{p^\pm} q_2^\pm\pm 2\gamma b^\pm (b^+-b^-)aq_{2,>\lambda}^\pm\\
    \nonumber
    &\gtrsim \mathbbm{1}_{|\xi|\geq \lambda}\mathbbm{1}_{W^\pm} c_j 2^{-j}|\xi|^{2},\qquad |x|\approx 2^j.
\end{align} Recall that $V_R^\pm\cup W^\pm=T^*\R^3\setminus o$. Combining (\ref{highfreqp1}) and (\ref{highfreqp2}) together, we conclude that 
$$(H_p q+2\gamma\tau aq)\big|_{\tau=b^\pm}\gtrsim \mathbbm{1}_{|\xi|\geq\lambda}\langle x\rangle^{-2}|\xi|^{2},$$ where we have used the slowly-varying, summable nature of $(c_j)$. 

This provides the desired bound over the characteristic set. To extend it to all of phase space, we must construct a lower-order correction term. 
Explicitly, we seek an $m\in S^0$ so that $$H_p q+2\gamma\tau aq+mp\gtrsim \mathbbm{1}_{|\xi|\geq\lambda}\inprod{x}^{-2}|\xi|^2.$$ If we write $$H_pq+2\gamma a\tau q=a_0\tau^2+a_1\tau+a_2,$$ where $a_j\in S^j$, then we have already established that \begin{align}\label{lang pos}
    a_0(x,\xi)(b^\pm(x,\xi))^2+a_1(x,\xi)b^\pm(x,\xi)+a_2(x,\xi)\gtrsim\mathbbm{1}_{|\xi|\geq\lambda}\inprod{x}^{-2}|\xi|^2
\end{align}  So, we must analyze the quantity   
$$a_0\tau^2+a_1\tau+a_2+p{m}=(a_0-{m})\tau^2+(a_1+(b^++b^-){m})\tau+(a_2-b^+b^-{m}).$$  If we choose ${m}$ so that 
\begin{align}
\label{cor cond 1}
a_0-{m}>0,\qquad |\xi|\geq\lambda
\end{align} and 
\begin{align}
\label{cor cond 2}
(a_1+(b^++b^-){m})^2-4(a_0-{m})(a_2-b^+b^-{m})<0,\qquad |\xi|\geq\lambda,
\end{align} then we will have that $a_0\tau^2+a_1\tau+a_2+mp$ is positive for $|\xi|\geq\lambda$ (the first condition on ${m}$ guarantees that this polynomial in $\tau$ is concave up, and the second guarantees that there are no 
real zeros). 

Let us begin by focusing on (\ref{cor cond 2}). The function \begin{align*}
P({m})&=(a_1+(b^++b^-){m})^2-4(a_0-{m})(a_2-b^+b^-{m})\\
&=(b^+-b^-)^2{m}^2+(2a_1(b^++b^-)+4a_0b^+b^-+4a_2){m}+(a_1^2-4a_0a_2)
\end{align*} is a quadratic polynomial in ${m}$ with a positive coefficient on the quadratic term, so it will achieve a minimal value at $${m}=-\frac{a_1(b^++b^-)+2(a_0b^+b^-+a_2)}{(b^+-b^-)^2}.$$ 

It is readily seen that ${m}\in S^0$ and that $m$ is supported where $|\xi|\geq\lambda$. This minimal value is 
\begin{align*}
P({m})&=\left(a_1-(b^+-b^-)\frac{a_1(b^++b^-)+2(a_0b^+b^-+a_2)}{(b^+-b^-)^2}\right)^2\\
&\qquad-4\left(a_0+\frac{a_1(b^++b^-)+2(a_0b^+b^-+a_2)}{(b^+-b^-)^2}\right)\left(a_2+b^+b^-\frac{a_1(b^++b^-)+2(a_0b^+b^-+a_2)}{(b^+-b^-)^2}\right)\\
&=-4\frac{(a_0(b^+)^2+a_1b^++a_2)(a_0(b^-)^2+a_1b^-+a_2)}{(b^+-b^-)^2}\\
&=-4(b^+-b^-)^{-2}\left((H_p q+2\gamma\tau aq)\big|_{\tau=b^+}\right)\left(H_p q+2\gamma\tau aq)\big|_{\tau=b^-}\right)\\
&<0,
\end{align*} where we have used (\ref{lang pos}).
So, (\ref{cor cond 2}) is satisfied. To establish (\ref{cor cond 1}), one can readily check that
 \begin{align*}
    a_0-{m}&=a_0+\frac{a_1(b^++b^-)+2(a_0b^+b^-+a_2)}{(b^+-b^-)^2}\\
    &=(b^+-b^-)^{-2}\left((H_p q+2\gamma\tau aq)\big|_{\tau=b^+}+(H_p q+2\gamma\tau aq)\big|_{\tau=b^-}\right)\\
    &>0
\end{align*} for $|\xi|\geq\lambda$.

This gives us that $$H_pq+2\gamma \tau aq+mp>0$$ for $|\xi|\geq\lambda$. In fact, we can check that the minimal value of the above in $\tau$ for $|\xi|\geq\lambda$  is 
   $$ \frac{(a_0(b^+)^2+a_1b^++a_2)(a_0(b^+)^2+a_1b^++a_2)}{(a_0(b^+)^2+a_1b^++a_2)+(a_0(b^-)^2+a_1b^{-}+a_2)}
    =\frac{\left((H_p q+2\gamma\tau aq)\big|_{\tau=b^+}\right)\left(H_p q+2\gamma\tau aq)\big|_{\tau=b^-}\right)}{(H_p q+2\gamma\tau aq)\big|_{\tau=b^+}+(H_p q+2\gamma\tau aq)\big|_{\tau=b^-}}.$$
 The numerator is bounded below by $\inprod{x}^{-4}|\xi|^4.$ In view of the support and symbolic properties of $q$, the denominator satisfies the bounds $$ (H_p q+2\gamma\tau aq)\big|_{\tau=b^+}+(H_p q+2\gamma\tau aq)\big|_{\tau=b^-}\approx \inprod{x}^{-2}|\xi|^2.$$ Since $|b^\pm(x,\xi)|\approx |\xi|$ and $|\tau|=|b^\pm(x,\xi)|$ in the above, we conclude the desired result.
\end{proof}

\subsection{Starting Energy Estimates and Case Reductions}
\label{reductions section}
In this section, we will establish various useful
energy estimates, then reduce the proof of Theorem \ref{high freq est} to a simpler problem. Our starting point is a
standard uniform energy inequality.
\begin{prop}\label{first ee}
Let $P$ be a stationary damped wave operator, $\partial_t$ be uniformly time-like, and $T>0$. Then, we have the estimate \begin{align}\label{starting ee}
    \norm{\partial u(t)}_{ L^2}^2\lesssim\norm{\partial u(0)}_{L^2}^2+ \int\limits_0^T\int\limits_{\R^3} |Pu\ \partial_t u|\, dxdt,\qquad 0\leq t\leq T
\end{align} for all $u\in \mathcal{W}_T$.
\end{prop}
\begin{proof}
Call $Pu=f$, and define the energy functional
 $$E[u](t)=\int\limits_{\mathbb{R}^3}D_ig^{ij}D_ju\overline{ u}-g^{00}|\partial_t u|^2\, dx.$$
 After integrating the first term by parts, this functional is readily seen to be coercive due to the uniformly time-like nature of $\partial_t,$ i.e. $$E[u](t)\approx \norm{\partial u(t)}_{L^2}^2.$$ Differentiating in $t$ and integrating by parts, 
gives that
\begin{align*}
     \frac{d}{dt}E[u](t)&=-\int\limits_{\R^3}g^{00}(\partial_t^2 u\partial_t\bar{u}+\partial_t u\partial_t^2\bar{u})\, dx+\int\limits_{\R^3}D_ig^{ij}D_j\partial_t u\bar{u}+D_ig^{ij}D_j u\partial_t\bar{u}\, dx\\
     &=\int\limits_{\R^3}(g^{00}D_t^2+D_ig^{ij}D_j)u\partial_t\bar{u}+\partial_t u\overline{(g^{00}D_t^2+D_ig^{ij}D_j)u}\, dx\\
     &=\int\limits_{\R^3}\left(-(g^{0j}D_jD_t+D_jg^{0j}D_t+iaD_t)u+f\right)\partial_t\bar{u}\\
     &\qquad+\partial_t u\overline{\left(-(g^{0j}D_jD_t+D_jg^{0j}D_t+iaD_t) u+f\right)}\, dx\\
     &=2\text{Re}\int\limits_{\R^3}\bar{f}\partial_t u\, dx-2\int\limits_{\R^3}a|\partial_tu|^2\, dx.
 \end{align*}

Dropping the damping term and integrating the resulting estimate in time yields the inequality  $$E[u](t)\lesssim E[u](0)+\int\limits_0^T\int\limits_{\R^3} |{f}\partial_t u|\, dxdt,\qquad 0\leq t\leq T.$$ Applying the coercivity allows us to conclude.
\end{proof}
\begin{remark}
Note that when $Pu=0,$ we have the energy dissipation statement $$ \frac{d}{dt}E[u](t)=-2\int\limits_{\R^3}a|\partial_tu|^2\, dx\leq 0.$$ 
\end{remark}
Applying the Schwarz inequality to Proposition \ref{first ee} provides us with estimates which will prove useful throughout this work.
\begin{Cor} \label{uniform ee}
Under the same assumptions as Proposition \ref{first ee}, the uniform energy estimates
 $$\norm{\partial u}_{L^\infty_t L^2_x}\lesssim \norm{\partial u(0)}_{L^2}+\norm{Pu}_{L^1_tL^2_x},$$  $$\norm{\partial u}_{L^\infty_t L^2_x}\lesssim \norm{\partial u(0)}_{L^2}+\norm{Pu}_{LE^*}^{1/2}\norm{u}_{LE^1}^{1/2},$$ and $$\norm{\partial u}_{L^\infty_t L^2_x}\lesssim \norm{\partial u(0)}_{L^2}+\varepsilon^{-1}\norm{Pu}_{LE^*+L^1_tL^2_x}+\varepsilon\norm{u}_{LE^1},\qquad\forall \varepsilon>0$$
 hold. 
\end{Cor}
\begin{proof} By Proposition \ref{first ee}, we have the estimate $$\norm{\partial u(t)}_{ L^2}^2\lesssim\norm{\partial u(0)}_{L^2}^2+ \int\limits_0^T\int\limits_{\R^3} |Pu\partial_t u|\, dxdt,\qquad 0\leq t\leq T.$$ To obtain the first estimate that we claimed, one applies the Schwarz inequality, takes a supremum in time of $\partial_t u$, and uses Young's inequality for products. 
To obtain the second estimate, write $$|Pu\ \partial_t u|=\left(\inprod{x}^{1/2}|Pu|\right)\left(\inprod{x}^{-1/2} |\partial_t u|\right)$$ in (\ref{starting ee}) and apply the Schwarz inequality and H\"older's inequality applied to $\ell^1$ with conjugate exponents $(p,q)=(1,\infty)$. 
\end{proof} 

Next, we cite an exterior estimate from \cites{mst20} (see \cites{MMT08} for a similar result for the Schr\"odinger equation).  Outside of a large enough compact spatial set, the operator $P$ is a small $AF$ perturbation of $\Box$, in which case we obtain good energy estimates with a necessary truncation error (here, we must measure the energy at time $T$, which is non-problematic in view of Corollary \ref{uniform ee}). 
\begin{prop}[Proposition 3.2 in \cites{mst20}] \label{exterior ee 2}
If $P$ is asymptotically flat and $R\geq R_0$, then \begin{align}\label{exterior ee 2 eqn} \norm{u}_{LE^1_{>R}}\lesssim\norm{\partial u(0)}_{L^2_{>R}}+\norm{\partial u(T)}_{L^2_{>R}}+R^{-1}\norm{u}_{LE_R}+\norm{Pu}_{LE^*_{>R}}.\end{align}
\end{prop} Since the damping is identically zero in this region, the result holds without any modification to the proof given in \cites{mst20}. For this reason, we will omit the details here, although we will provide the overarching idea: Their proof is a positive commutator argument using the multiplier $Q_1+Q_2$, where $$Q_1=\chi_{>2R}(|x|)f(|x|)\frac{x_j}{|x|}g^{jk} D_k+D_k\chi_{>2R}(|x|)f(|x|)\frac{x_j}{|x|}g^{jk}$$ is the principal term, and  $$Q_2=\chi_{>2R}(|x|)f'(|x|)$$ is the lower-order correction term. Here, $ f(|x|)=\displaystyle{\frac{|x|}{|x|+2^j}}$, and $j$ is chosen so that $2^j\geq R.$

Before moving on to proving Theorem \ref{high freq est}, we will simplify its proof through case reductions. As in Section 4 of \cites{mst20}, one can readily reduce to the case of $u$ having zero Cauchy data at times $0$ and $T$, as well as $f\in LE^*_c$. To do this, one constructs an approximate solution to a problem with the same data and forcing with a wave operator which is a small $AF$ perturbation of $\Box$ and agrees with $P$ for $|x|>R_0,$ then considers approximations using a unit time interval partition of unity and matching the initial (respectively final time) Cauchy data on the first
(respectively last) solution granted by the partition. In view of these reductions, it is enough to establish \begin{equation}\label{goal est}
     \norm{u}_{LE^1[0,T]}\lesssim \norm{\inprod{x}^{-2} u}_{LE[0,T]}+\norm{Pu}_{LE_c^*[0,T]}
\end{equation} for all $u\in\mathcal{W}_T$ satisfying that $u[0]=u[T]=0$ in order to prove Theorem \ref{high freq est}. The implicit constant in (\ref{goal est}) is still independent of $T$. 

We will explicitly establish an additional reduction, namely to solutions with compact spatial support, using Proposition \ref{exterior ee 2}.
\vskip .1in
 \noindent\textbf{Claim.} It suffices to prove (\ref{goal est}), and hence Theorem \ref{high freq est}, for $u$ supported in $\{|x|\leq   2R_0\}$. 
 \begin{proof}
Write  $u=\chi_{< R_0}u+\chi_{>R_0}u.$ On the exterior piece $\chi_{>R_0}u$, we apply Proposition \ref{exterior ee 2} and Corollary \ref{uniform ee} to get that \begin{align*}
    \norm{\chi_{>R_0}u}_{LE^1[0,T]}\lesssim R_0^{-1}\norm{\chi_{>R_0}u}_{LE_{R_0}[0,T]}+\norm{P(\chi_{>R_0}u)}_{LE^*_{>R_0}[0,T]}.\end{align*}
The first term on the right is directly bounded by $\norm{u}_{LE^1_{R_0}[0,T]}.$ 
For the second term, we write $$P(\chi_{>R_0}u)=(\chi_{>R_0})Pu+[P,\chi_{>R_0}]u,$$ and one can calculate that 
$$[P,\chi_{>R_0}]u(t,x)=\mathcal{O}(R_0^{-1})\chi'\left(\frac{|x|}{R_0}\right)\partial u(t,x)+\mathcal{O}(R_0^{-2})\chi''\left(\frac{|x|}{R_0}\right)u(t,x).$$  
In $LE^*$, this term bounded by $\norm{u}_{LE^1_{R_0\leq |\cdot|\leq 2R_0}[0,T]}$, and so 
we have $$\norm{\chi_{>R_0}u}_{LE^1[0,T]}\lesssim \norm{u}_{LE^1_{R_0}[0,T]}.$$
Suppose that (\ref{goal est}) holds for $\chi_{< R_0}u.$ From this, we get the estimate $$\norm{\chi_{< R_0} u}_{LE^1[0,T]}\lesssim \norm{\inprod{x}^{-2}\chi_{< R_0}u}_{LE[0,T]}+\norm{Pu}_{LE^*_c[0,T]}+\norm{[P,\chi_{< R_0}]u}_{LE^*_c[0,T]},$$ and so 
\begin{align*}
\norm{u}_{LE^1[0,T]}&\leq \norm{\chi_{< R_0}u}_{LE^1[0,T]}+\norm{\chi_{>R_0}u}_{LE^1[0,T]}\\
&\lesssim \norm{\inprod{x}^{-2} u}_{LE[0,T]}+\norm{Pu}_{LE^*_c[0,T]}+\norm{u}_{LE^1_{R_0\leq |\cdot|\leq 2R_0}[0,T]}.
\end{align*} The last term is readily estimated via Proposition \ref{exterior ee 2}, which establishes (\ref{goal est}). 
 \end{proof} 

 We record the results of these case reductions in the following proposition.
\begin{prop}\label{case red prop}
 In order to establish Theorem \ref{high freq est}, it is sufficient to prove the estimate 
 \begin{align}\label{simplified est}
     \norm{v}_{LE^1[0,T]}\lesssim \norm{v}_{L^2_tL^2_x[0,T]}+\norm{Pv}_{LE^*[0,T]}
 \end{align} for $v$ supported in  $\{|x|\leq 2R_0\}$ with $v[0]=v[T]=0.$
\end{prop}
Again, this implicit constant is independent of $T$ but will depend on $R_0$. Using the compact support of $v$ to transition between the weighted and unweighted spaces will inherently generate multiplication by powers of $R_0$, but this does not matter since the constant in the above may depend on such a parameter.

\subsection{Proof of the High Frequency Estimate}
\label{thm section}
Armed with the established case reductions, we will proceed with a proof of Theorem \ref{high freq est}. Recall that it is equivalent to analyze the scaled problem. 

\begin{proof}[Proof of Theorem \ref{high freq est}] 
We will break this proof into a sequence of steps. 

\vskip .1in
\noindent \textbf{Step 1: Setting up the Positive Commutator and the Frequency Decomposition}. First, we remark that in view of Proposition \ref{case red prop}, it will suffice to prove (\ref{simplified est}) for $v$ supported in $\{|x|<2R_0\}$ with $v[0]=v[T]=0.$ We can extend $v$ by zero to be defined for $t\in\R$ and vanish for $t\notin (0,T).$ Then,
\begin{align}\label{comm}  2\text{Im}
\inprod {P v, \left(q^{\operatorname{w}}-\frac{i}{2}m^{\operatorname{w}}\right)v} +\frac{i\gamma }{2}\inprod{ [a D_t, m^{\operatorname{w}}]v,v}&=\inprod{i[\Box_{g},q^{\operatorname{w}}]v,v}+\gamma\inprod{(q^{\operatorname{w}}a  D_t+a D_t q^{\operatorname{w}})v,v}\\
&\qquad\qquad+\frac{1}{2}\inprod{(\Box_{g}m^{\operatorname{w}}+m^{\operatorname{w}}\Box_{g})v,v}\nonumber. \end{align} 
The right-hand side of (\ref{comm}) can be written as 
\begin{align*}
   \inprod{i[\Box_{g},q^{\operatorname{w}}]v,v}&+\gamma\inprod{(q^{\operatorname{w}}a  D_t+a D_t q^{\operatorname{w}})v,v}+\frac{1}{2}\inprod{(\Box_{g}m^{\operatorname{w}}+m^{\operatorname{w}}\Box_{g})v,v} \\
   &=\inprod{(H_{p}q+2\gamma\tau a q+mp)^{\operatorname{w}} v,v}+\inprod{ A_0 v,v},
\end{align*}
where $A_0\in \Psi^0$. Recall that $2\gamma\tau a=-2is_{skew}.$ By choosing $\gamma>0$ large enough, we can apply Lemma \ref{escape} to get \begin{align}\label{pre gard} H_{p}q-2is_{skew} q+mp-C\mathbbm{1}_{|\xi|\geq\lambda}\langle x\rangle^{-2}(|\xi|^2+\tau^2)\geq 0,
\end{align} 
where $C>0$ is the implicit constant in Lemma \ref{escape}. We can readily replace $\mathbbm{1}_{|\xi|\geq \lambda}$ with $\chi_{|\xi|>\lambda}$. 

We will split $v$ into frequency components via $v=v_{>>\lambda}+v_{<<\lambda}$, where
\begin{align*}   v_{>>\lambda}&=\chi_{|\xi|+|\tau|>\lambda}(\partial)v,\\
    v_{<<\lambda}&=\chi_{|\xi|+|\tau|<\lambda}(\partial)v.
\end{align*}  Since the desired estimate is a high frequency estimate, we will first analyze the high frequency components of $v$.

\vskip .1in
\noindent \textbf{Step 2a: High Frequencies - Applying the G\r{a}rding inequality}. By (\ref{pre gard}), we may apply the sharp G\r{a}rding inequality to obtain that \begin{align*}\inprod{(H_{p}q-2is_{skew} q+mp)^{\operatorname{w}} v_{>>\lambda},v_{>>\lambda}}\gtrsim \inprod{\left(\chi_{|\xi|>\lambda}\langle x\rangle^{-2}(|\xi|^2+\tau^2)\right)^{\operatorname{w}} v_{>>\lambda},v_{>>\lambda}}-\norm{v_{>>\lambda}}_{H^{1/2}_{t,x}}^2.
\end{align*}
We remark that the implicit constant may be chosen independently of $\lambda$ since there is a $\chi_{|\xi|>\lambda}$ cutoff embedded into $q$ and $m$, and hence differentiation occurring in asymptotic expansion calculations possess coefficients which are either independent of $\lambda$ or feature inverse powers of $\lambda$ (one can also entirely ignore the potential $\lambda$ dependence and argue via Cauchy-Schwarz and Young's inequality for products, although this introduces more parameters to track). 

Since $\chi_{|\xi|+|\tau|<\lambda}\in S^{-\infty}$, it follows that $$\inprod{(H_{p}q-2is_{skew} q+mp)^{\operatorname{w}} v,v}=\inprod{(H_{p}q-2is_{skew} q+mp)^{\operatorname{w}} v_{>>\lambda},v_{>>\lambda}}+\inprod{S_0v,v},$$ where $S_0\in\Psi^{-\infty}.$ In particular,
\begin{align}
\label{gard app}
\inprod{(H_{p}q-2is_{skew} q+mp)^{\operatorname{w}} v,v}\gtrsim \inprod{\left(\chi_{|\xi|>\lambda}\langle x\rangle^{-2}(|\xi|^2+\tau^2)\right)^{\operatorname{w}} v_{>>\lambda},v_{>>\lambda}}-\norm{v_{>>\lambda}}_{H^{1/2}_{t,x}}^2+\inprod{S_0v,v}.
\end{align}

Using the pseudodifferential composition formula, we compute that \begin{align}\label{inter gard}
\left(\chi_{|\xi|>\lambda}\langle x\rangle^{-2}(|\xi|^2+\tau^2)\right)^{\operatorname{w}}=(\chi_{|\xi|>\lambda}(D_x))^{1/2}D_\alpha\langle x\rangle^{-2} D_\alpha (\chi_{|\xi|>\lambda}(D_x))^{1/2}+A_1,\end{align} where $A_1\in \Psi^1$ arises from non-principal terms in the asymptotic expansion of the Moyal product (and the expansion features terms which are either independent of $\lambda$ or involve inverse powers of $\lambda$). 
Integrating by parts once gives that 
\begin{align}\label{gard err}
\inprod{\left(\chi_{|\xi|>\lambda}\langle x\rangle^{-2}(|\xi|^2+\tau^2)\right)^{\operatorname{w}}v_{>>\lambda},v_{>>\lambda}}&=\norm{\inprod{x}^{-1}\partial v_{>\lambda}}_{L^2_tL^2_x}^2+\inprod{{A}_1v_{>>\lambda},v_{>>\lambda}}\\
\nonumber
&\gtrsim\norm{\partial v_{>\lambda}}_{LE_{<2R_0}}^2+\inprod{{A}_1v_{>>\lambda},v_{>>\lambda}},
\end{align} 
where $v_{>\lambda}=\chi_{|\xi|>\lambda}(D_x)v$ and $A_1\in\Psi^1$ is now a modification of the previous version of the same variable in (\ref{inter gard}) (to-be-explained momentarily). One might expect the term  $\partial\left((\chi_{|\xi|>\lambda}(D_x))^{1/2})v_{>>\lambda}\right)$ to appear instead of $\partial v_{>\lambda}$, but it is readily seen that $$\left(\chi_{|\xi|>\lambda}(|\xi|)\right)^{1/2}\chi_{|\xi|+|\tau|>\lambda}(|(\tau,\xi)|)\approx \chi_{|\xi|>\lambda}(|\xi|)\chi_{|\xi|+|\tau|>\lambda}(|(\tau,\xi)|)=\chi_{|\xi|>\lambda}(|\xi|).$$ In particular, the $\tau$ has no effect on the resulting cutoff, and $\chi,\chi^{1/2}$ are both smooth, non-decreasing, and have the same support properties (and only differ on a compact set). The only effect in exchanging these terms is modifying $A_1$ in order the errors resulting from this switch; hence, the $A_1$ in (\ref{gard err}) is different than in (\ref{inter gard}). For this reason, none of our analysis changes by working with $v_{>\lambda}$, and we will stick with this for notational convenience.

After incorporating (\ref{gard err}) into (\ref{gard app}), we have  that
\begin{align}\label{rhs gard err}
    \inprod{(H_{p}q-2is_{skew} a q+mp)^{\operatorname{w}} v,v}+\inprod{A_0 v,v}&\gtrsim\norm{\partial v_{>\lambda}}_{LE_{<2R_0}}^2 
    -\norm{v_{>>\lambda}}_{H^{1/2}_{t,x}}^2\\  \nonumber 
    &\qquad-\left|\inprod{A_1 v_{>>\lambda},v_{>>\lambda}}\right|-\left|\inprod{A_0 v,v}\right|-\left|\inprod{S_0 v,v}\right|.
\end{align}

\vskip .1in
\noindent \textbf{Step 2b: High Frequencies - Handling the Error Terms in (\ref{rhs gard err})}. We will first analyze the term $\inprod{A_1 v_{>>\lambda},v_{>>\lambda}}.$  Since $A_1\in \Psi^1,$ it is bounded from $H^1_{t,x}$ to $L^2_tL^2_x$ (and the operator norm will yield no positive-power $\lambda$ contributions due to the previous comment on the asymptotic expansion of the symbol). By using the Schwarz inequality and this mapping property, we have that  
\begin{align}\label{order 1 pdo}\left|\inprod{A_1 v_{>>\lambda},v_{>>\lambda}}\right|\lesssim\norm{v_{>>\lambda}}_{H^1_{t,x}}\norm{v_{>>\lambda}}_{L^2_tL^2_x}.
\end{align}
Using Plancherel's theorem in $(t,x)$, the frequency localization, and the compact support of $v$, we obtain the bounds
\begin{align}\label{order 1 pdo 1}
\norm{v_{>>\lambda}}_{H^1_{t,x}}&\lesssim \norm{\inprod{(\tau,\xi)}\chi_{|\xi|+|\tau|>\lambda}\hat{v}}_{L^2_\tau L^2_\xi}\lesssim \norm{\inprod{(\tau,\xi)}\hat{v}}_{L^2_\tau L^2_\xi}=\norm{v}_{H^1_{t,x}}\lesssim\norm{v}_{LE^1},
\end{align} and
\begin{align}\label{order 1 pdo 2}
   \norm{v_{>>\lambda}}_{L^2_tL^2_x}&\approx\norm{\chi_{|\xi|+|\tau|>\lambda}\hat{v}}_{L^2_\tau L^2_\xi}\lesssim \norm{\frac{|\tau|+|\xi|}{\lambda}\chi_{|\xi|+|\tau|>\lambda}\hat{v}}_{L^2_\tau L^2_\xi}\lesssim\lambda^{-1}\norm{\partial v}_{L^2_tL^2_x}\lesssim\lambda^{-1}\norm{v}_{LE^1}.
\end{align} 

Applying (\ref{order 1 pdo 1}) and (\ref{order 1 pdo 2}) to (\ref{order 1 pdo}) yields that
$$ \left|\inprod{A_1 v_{>>\lambda},v_{>>\lambda}}\right|\lesssim \lambda^{-1}\norm{v}_{LE^1}^2.$$ 
For the term $\norm{v_{>>\lambda}}_{H^{1/2}_{t,x}}^2$, note that
\begin{align*}
    \norm{v_{>>\lambda}}_{H^{1/2}_{t,x}}^2&\lesssim \norm{\inprod{(\tau,\xi)}^{1/2}\chi_{|\xi|+|\tau|>\lambda}\hat{v}}_{L^2_\tau L^2_\xi}^2=\norm{\inprod{(\tau,\xi)}^{-1/2}\inprod{(\tau,\xi)}\chi_{|\xi|+|\tau|>\lambda}\hat{v}}_{L^2_\tau L^2_\xi}^2\\
    &\lesssim \lambda^{-1}\norm{\inprod{(\tau,\xi)}\chi_{|\xi|+|\tau|>\lambda}\hat{v}}_{L^2_\tau L^2_\xi}^2\lesssim\lambda^{-1}\norm{v}_{LE^1}^2.
\end{align*} 
For the $\inprod{A_0v,v}$ term, we can use $L^2$-boundedness and the compact support of $v$ to get  
\begin{align}\label{zero order pdo}
\left|\inprod{A_0 v,v}\right|\leq\norm{A_0v}_{L^2_tL^2_x}\norm{v}_{L^2_tL^2_x}\lesssim C(\lambda)  \norm{v}_{L^2_tL^2_x}^2.
\end{align} While this bound is $\lambda$-dependent, such terms appear on the \textit{upper bound side} of the desired inequality, and hence can depend on $\lambda$ in an arbitrary manner (as opposed to the $LE^1$ terms which need an inverse power of $\lambda$ for bootstrapping). The meaning of $C(\lambda)$ will change fluidly, just as one continuously re-notates a potentially-changing constant by $C$ when calculating successive inequalities.

The smoothing term $\inprod{S_0v,v}$ can be bounded in the same way as $\inprod{A_0v,v}$ (in particular, $S_0\in \Psi^0$). Thus, we have the lower bound
\begin{align}\label{pos com 1}
    \inprod{(H_{p}q-2is_{skew} q+mp)^{\operatorname{w}} v,v}+\inprod{A_0 v,v}&\gtrsim \norm{\partial v_{>\lambda}}_{LE_{<2R_0}}^2- C(\lambda)\norm{v}_{L^2_t L^2_x}^2
    -\lambda^{-1}\norm{v}_{LE^1}^2.
\end{align} Next, we look at the left-hand side of (\ref{comm}). 
\vskip .1in
\noindent \textbf{Step 3: Bounding the Left-Hand Side of (\ref{comm})}.
Since $[aD_t,m^{\operatorname{w}}]\in \Psi^0$, performing the same work as in (\ref{zero order pdo}) provides that 
\begin{align}\label{pos com 2}
\left|\frac{i\gamma}{2}\inprod{[a D_t, m^{\operatorname{w}}]v,v}\right|\lesssim C(\lambda)\norm{v}_{L^2_tL^2_x}^2.
\end{align} For remaining term on the left-hand side of (\ref{comm}), we split $v$ into high and low frequency components once again to get that $$2\text{Im}
\inprod {P v, \left(q^{\operatorname{w}}-\frac{i}{2}m^{\operatorname{w}}\right)v}= 2\text{Im}
\inprod {P v, \left(q^{\operatorname{w}}-\frac{i}{2}m^{\operatorname{w}}\right)v_{>>\lambda}}+\inprod{S_1 v,v},$$
where $S_1\in \Psi^{-\infty}.$ We have already demonstrated how to bound smoothing operator terms. For the other (primary) piece, we apply the Schwarz inequality, use the $\Psi$DO mapping properties of $q^{\operatorname{w}}\in \Psi^1$ and $m^{\operatorname{w}}\in \Psi^0$, and leverage the compact support of $v$ (just as performed previously) to get 
\begin{align} \label{pos com 3}
    \left|2\text{Im}
\inprod {P v, \left(q^{\operatorname{w}}-\frac{i}{2}m^{\operatorname{w}}\right)v_{>>\lambda}}\right|\lesssim C(\lambda)\norm{P v}_{L^2_tL^2_x}\norm{v}_{LE^1}\lesssim C(\lambda) \norm{P v}_{LE^*_c}\norm{v}_{LE^1}.
\end{align}

\vskip .1in
\noindent \textbf{Step 4: Combining the Established Bounds Into a High Frequency Bound}.
Putting (\ref{comm}), (\ref{pos com 1}), (\ref{pos com 2}), and (\ref{pos com 3}) together, we obtain that   
\begin{align*}
    \norm{\partial v_{>\lambda}}_{LE_{<2R_0}}&\lesssim C(\lambda)\left(\norm{P v}_{LE^*}^{1/2}\norm{v}_{LE^1}^{1/2}+\norm{v}_{L^2_t L^2_x}\right)+\lambda^{-1/2}\norm{ v}_{LE^1}.
\end{align*} Completing the $LE^1_{<2R_0}$ norm on the left-hand side of the above,
\begin{align*}
    \norm{ v_{>\lambda}}_{LE^1_{<2R_0}}&\lesssim C(\lambda)\left(\norm{P v}_{LE^*}^{1/2}\norm{v}_{LE^1}^{1/2}+\norm{v}_{L^2_tL^2_x}\right)+\lambda^{-1/2}\norm{v}_{LE^1}+\norm{\inprod{x}^{-1}v_{>\lambda}}_{LE}.
\end{align*} We note that $$\norm{\inprod{x}^{-1}v_{>\lambda}}_{LE}\lesssim \norm{v}_{L^2_tL^2_x},$$ once again using Plancherel's theorem.  Thus,   
\begin{align}\label{int high freq}
\norm{ v_{>\lambda}}_{LE^1_{<2R_0}}\lesssim C(\lambda)\left(\norm{P v}_{LE^*}^{1/2}\norm{v}_{LE^1}^{1/2}+\norm{v}_{L^2_t L^2_x}\right)+\lambda^{-1/2}\norm{v}_{LE^1}.    
\end{align} This establishes an estimate on the high frequencies. We must add in the lower frequencies to the left-hand side. That is, we must add $\norm{ v_{<\lambda}}_{LE^1_{<2R_0}}$ to both sides of (\ref{int high freq}).

\vskip .1in
\noindent \textbf{Step 5a: Lower Frequencies - Further Frequency Splitting and Bounding the Low-Low Term}.
 First, we get the bound
$$\norm{\inprod{x}^{-1}v_{<\lambda}}_{LE}\lesssim \norm{v}_{L^2_tL^2_x} $$
 via Plancherel's theorem.
For the term $\norm{\partial v_{<\lambda}}_{LE}$, we write 
\begin{equation*}v_{<\lambda}=v_{<>\sigma\lambda}+v_{<<\sigma\lambda},
\end{equation*}
 where 
\begin{align*}
    v_{<>\sigma\lambda}&=\chi_{|\xi|<\lambda}(D_x)\chi_{|\tau|> \sigma\lambda}(D_t) v,\\
    v_{<<\sigma\lambda}&=\chi_{|\xi|<\lambda}(D_x)\chi_{|\tau|<\sigma\lambda}(D_t) v,
\end{align*}
and 
$\sigma\gg 1$ will be chosen later (and does not denote the same $\sigma$ as used in the construction of the escape function). Applying Plancherel's theorem, frequency localization, and the compact support of $v$ again yields
\begin{align*}
    \norm{\partial v_{<<\sigma\lambda}}_{LE}\lesssim \norm{(|\tau|+|\xi|)\chi_{|\xi|<\lambda}\chi_{|\tau|<\sigma\lambda}\hat{v}}_{L^2_\tau L^2_\xi}\lesssim\sigma\lambda\norm{v}_{L^2_tL^2_x}.
\end{align*}

\vskip .1in
\noindent \textbf{Step 5b: Lower Frequencies - Bounding the Low-High Term}.
 For $v_{<>\sigma\lambda},$ we compute that 
\begin{align} \label{mixed freq piece}
    \norm{\partial v_{<>\sigma\lambda}}_{LE}\lesssim \norm{(|\tau|+|\xi|)\chi_{|\xi|<\lambda}\chi_{|\tau|>\sigma\lambda}\hat{v}}_{L^2_\tau L^2_\xi}\lesssim\lambda\norm{v}_{L^2_t L^2_x}+(\sigma\lambda)^{-1}\norm{(\partial_t^2v)_{<>\sigma\lambda}}_{L^2_t L^2_x}.
\end{align} For the last term on the right, we utilize the expression for $Pv$ to write 
\begin{align}\label{low freq cont err}
    \norm{(\partial_t^2v)_{<>\sigma\lambda}}_{L^2_t L^2_x}&\lesssim \norm{(Pv)_{<>\sigma\lambda}}_{L^2_tL^2_x}+\norm{\left((g^{0j}D_j+D_jg^{0j})D_t v\right)_{<>\sigma\lambda}}_{L^2_tL^2_x}\\
 \nonumber  &\qquad+\norm{(D_ig^{ij}D_jv)_{<>\sigma\lambda}}_{L^2_tL^2_x}+\norm{(aD_tv)_{<>\sigma\lambda}}_{L^2_tL^2_x}.
\end{align}
One can readily check that \begin{align}\label{low freq cont 1}
\norm{(Pv)_{<>\sigma\lambda}}_{L^2_tL^2_x}\lesssim\norm{Pv}_{LE^*},
\end{align}
and \begin{align}\label{low freq cont 2}
\norm{(aD_tv)_{<>\sigma\lambda}}_{L^2_tL^2_x}\lesssim \norm{\partial v}_{LE}.
\end{align} For the other terms, we note that as \textit{functions}, one has that $g^{\alpha j}, D_jg^{\alpha j}\in S^0$ for all $\alpha\in \{0,1,2,3\}$ and $j\in \{1,2,3\}$, and so $$[\chi_{|\xi|<\lambda}(D_x)\chi_{|\tau|>\sigma\lambda}(D_t),g^{\alpha j}]\in \Psi^{-1},\qquad
[\chi_{|\xi|<\lambda}(D_x)\chi_{|\tau|>\sigma\lambda}(D_t),D_jg^{\alpha j}]\in \Psi^{-1}.$$ 
In particular, the above two operators are bounded on $L^2_tL^2_x.$ Pairing this with the fact that Fourier multipliers commute, we have that
\begin{align} \label{low freq cont 3}
    \norm{\left((g^{0j}D_j+D_jg^{0j})D_t v\right)_{<>\sigma\lambda}}_{L^2_tL^2_x}&\lesssim \norm{(D_jg^{0j})\left(D_t v\right)_{<>\sigma\lambda}}_{L^2_tL^2_x}+\norm{g^{0j}\left(D_jD_t v\right)_{<>\sigma\lambda}}_{L^2_tL^2_x}\\
   \nonumber &\qquad+\norm{[\chi_{|\xi|<\lambda}(D_x)\chi_{|\tau|>\sigma\lambda}(D_t),(D_jg^{0 j})]D_t v}_{L^2_tL^2_x}\\
   \nonumber  &\qquad+\norm{[\chi_{|\xi|<\lambda}(D_x)\chi_{|\tau|>\sigma\lambda}(D_t),g^{0 j}]D_jD_t v}_{L^2_tL^2_x}\\
   \nonumber  &\lesssim \lambda \norm{\partial v}_{LE}+C(\lambda)\norm{v}_{L^2_tL^2_x},
\end{align}
and
\begin{align} \label{low freq cont 4}
    \norm{(D_ig^{ij}D_jv)_{<>\sigma\lambda}}_{L^2_tL^2_x}&\lesssim \norm{(D_ig^{ij})(D_j v)_{<>\sigma\lambda}}_{L^2_tL^2_x}+\norm{g^{ij}(D_iD_j v)_{<>\sigma\lambda}}_{L^2_tL^2_x}\\
   \nonumber  &\qquad+\norm{([\chi_{|\xi|<\lambda}(D_x)\chi_{|\tau|>\sigma\lambda}(D_t),(D_ig^{ij})](D_j v)_{<>\sigma\lambda}}_{L^2_tL^2_x}\\
   \nonumber  &\qquad+\norm{([\chi_{|\xi|<\lambda}(D_x)\beta_{|\tau|\geq\sigma\lambda}(D_t),g^{ij}](D_iD_j v)_{<>\sigma\lambda}}_{L^2_tL^2_x}\\
   \nonumber  &\lesssim C(\lambda)\norm{v}_{L^2_tL^2_x}.
\end{align} Applying (\ref{low freq cont 1})-(\ref{low freq cont 4}) to (\ref{low freq cont err}) gives that
\begin{align*}
\norm{(\partial_t^2v)_{<>\sigma\lambda}}_{L^2_t L^2_x}&\lesssim C(\lambda)\norm{v}_{L^2_tL^2_x}+\lambda\norm{\partial v}_{LE}+\norm{Pv}_{LE^*}.
\end{align*} Plugging the resulting estimate into (\ref{mixed freq piece}) implies that \begin{align*}\norm{\partial v_{<>\sigma\lambda}}_{LE}&\lesssim C(\lambda)\norm{v}_{L^2_tL^2_x}+(\sigma\lambda)^{-1}\norm{Pv}_{LE^*}+\sigma^{-1}\norm{\partial v}_{LE}.
\end{align*}

\vskip .1in
\noindent \textbf{Step 5c: Lower-Frequency - Combining All Lower Frequency Contributions}. Thus, the full low frequency contribution yields
\begin{align}\label{low freq cont full}
\norm{\partial v_{<\lambda}}_{LE}&\lesssim \max\{C(\lambda),\sigma\lambda\}\norm{v}_{L^2_tL^2_x}+(\sigma\lambda)^{-1}\norm{Pv}_{LE^*}+\sigma^{-1}\norm{\partial v}_{LE}\\
\nonumber &\lesssim\max\{C(\lambda),\sigma\lambda\}\norm{v}_{L^2_tL^2_x}+(\sigma\lambda)^{-1}\norm{Pv}_{LE^*}+\sigma^{-1}\norm{v}_{LE^1}.
\end{align}
\vskip .1in
\noindent \textbf{Step 6: Combining the High and Lower-Frequency Bounds}.
Now, we can combine the high frequency work (\ref{int high freq}) with the low frequency work (\ref{low freq cont full}) and apply Young's inequality for products with parameter $\delta>0$ to obtain that
\begin{align*}
    \norm{v}_{LE^1_{<2R_0}}&\lesssim C(\lambda)\norm{P v}_{LE^*}^{1/2}\norm{v}_{LE^1}^{1/2}+\max\{C(\lambda),\sigma\lambda\}\norm{v}_{L^2_tL^2_x}+(\sigma\lambda)^{-1}\norm{Pv}_{LE^*}+\left(\sigma^{-1}+\lambda^{-1/2}\right)\norm{ v}_{LE^1}\\
    &\lesssim\max\{C(\lambda),\sigma\lambda\}\norm{v}_{L^2_tL^2_x}+\left([C(\lambda)]^2\delta^{-1}+(\sigma\lambda)^{-1}\right)\norm{Pv}_{LE^*}+\left(\delta+\sigma^{-1}+\lambda^{-1/2}\right)\norm{v}_{LE^1}
\end{align*} Due to the support of $v$ in $x$, we know that $\norm{v}_{LE^1_{<2R_0}}=\norm{v}_{LE^1}$. Picking $\delta$ sufficiently small and $\lambda,\sigma$ sufficiently large (all of which will depend on $R_0$) allows us to absorb the $\norm{v}_{LE^1}$ term on the right-hand side into the left-hand side, providing (\ref{simplified est}) and completing the proof.
\end{proof}
 \section{Local Energy Decay}\label{LED section}
 In this section, we explain how recovering the high frequency estimate in \cites{mst20} allows us establish local energy decay by appealing to existing estimates in their work (which makes it easier to perform time frequency localization). By an extension argument outlined in Section \ref{ext arg}, it is sufficient to prove to reduce to the case of Schwartz functions, which allows for the removal of data terms. The simplified version of local energy decay for Schwartz functions can be readily proven by combining the proven high frequency estimates with appropriate medium and low frequency estimates then utilizing a time-frequency partition of unity. 
 The medium and low frequency estimates that we require come from the work in \cites{mst20} and do not depend on the trapping nor the damping. We remark that such analyses are also \textit{independent} of the stationarity of $P$. 

 \subsection{Medium Frequencies}\label{med freq}
The goal of the medium frequency estimate is to establish a weighted estimate which implies local energy decay for solutions supported at any range of time frequencies bounded away from both zero and infinity. This is rooted in the notion of a \textit{Carleman estimate}, which is weighted $L^2_tL^2_x$ estimates where the weight is pseudoconvex. 
The constants in our inequalities will depend on the parameter $\textbf{c}$ introduced in Section 1.2, but they will (and must) be independent of the parameters in $\varphi$; the Carleman weights for our estimates are radial. The Carleman weights which we use here are constructed in e.g. \cites{Bo18},  \cites{KT01}.  

The main medium frequency estimate is the following, and the corresponding theorem in \cites{mst20} is Theorem 5.4.
\begin{Th}\label{med freq est}
Let $P$ be an asymptotically flat damped wave operator, and suppose that $\partial_t$ is uniformly time-like. Then, for any $\delta>0$, there exists a bounded, non-decreasing radial weight $\varphi=\varphi(\ln (1+r))$ so that for all $u\in \mathcal{S}(\mathbb{R}^4)$, we have the bound 
\begin{multline}\label{carl4est}
    \norm{(1+\varphi_+'')^{1/2}e^\varphi(\nabla u,\inprod{r}^{-1}(1+\varphi')u}_{LE}+\norm{(1+\varphi')^{1/2}e^\varphi\partial_t u}_{LE}
   \\ \lesssim \norm{e^\varphi Pu}_{LE^*}
    +\delta\left(\norm{(1+\varphi')^{1/2}e^\varphi u}_{LE}+\norm{\inprod{r}^{-1}(1+\varphi_+'')^{1/2}(1+\varphi')e^\varphi \partial_t u}_{LE}\right).
\end{multline}
\end{Th}
\begin{remark}\label{med freq rem}
We will justify \textit{why} this is an appropriate estimate on the medium frequencies. To that end, suppose that $u$ is supported at time frequencies $\tau$ such that $0<\tau_0\leq |\tau|\leq \tau_1$, where $\tau_0<\tau_1$. For compatibility with the other frequency regimes, we will want $\tau_0\ll 1\ll \tau_1$. Plancherel's theorem yields that  $$\delta\norm{(1+\varphi')^{1/2}e^\varphi u}_{LE}\lesssim \frac{\delta}{\tau_0}\norm{(1+\varphi')^{1/2}e^\varphi \partial_t u}_{LE},$$ while $$\delta\norm{\inprod{r}^{-1}(1+\varphi_+'')^{1/2}(1+\varphi')e^\varphi \partial_t u}_{LE}\lesssim \delta\tau_1 \norm{\inprod{r}^{-1}(1+\varphi_+'')^{1/2}(1+\varphi')e^\varphi u}_{LE}.$$ By choosing $\delta$ sufficiently small, both terms absorb into the left-hand side of (\ref{carl4est}) in a direct fashion. We can translate our work immediately into a local energy decay estimate for $u$, with an implicit constant which depends on $\varphi$. Notice that since $\delta$ can be chosen arbitrarily, (\ref{carl4est}) allows for any interval of frequencies bounded away from both zero and infinity.
\end{remark}

The proof of this theorem is broken up into two Carleman estimates, one which applies within a large compact set and one which applies outside of this compact set. Within the compact set, the damping term is well-signed and readily absorbable as a perturbation due to the conditions on the weight $\varphi.$ Here, the weight will be convex. Outside of the compact set, the damping is zero, so the proof in \cites{mst20} follows through without any modification. In this region, one desires to use Proposition \ref{exterior ee 2}, which requires a constant weight. To that end, one breaks up the exterior into three regions: one where the Carleman weight is convex, a transition region where the conditions break in order to bend the weight to be constant near infinity, and a region near infinity where the weight is constant. 

In both regions (the compact set and its exterior), the work in \cites{mst20} allows for more general (e.g. unsigned) lower-order terms than a damping term. The proofs of these estimates within the aforementioned regions are based on positive commutator arguments utilizing the self- and skew-adjoint parts of the conjugated operator $P_\varphi=e^\varphi Pe^{-\varphi}.$ The work is then combined using a cutoff argument.
 \subsection{Low Frequencies}\label{low freq}
We will set  $$P_0:=P\big|_{D_t=0}=D_ig^{ij}D_j.$$ This represents $P$ at time frequency zero, and one can utilize it to obtain information in a neighborhood of this frequency.  Since $\partial_t$ is uniformly time-like, $P_0$ is uniformly elliptic. The operator $P_0$ is a special case of that found in \cites{mst20}, so all of the results in their work apply here with almost no modification. Notice that the damping does not arise in $P_0.$

At low frequencies, \textit{the} obstruction to local energy decay arises when $P$ has a \textit{resonance} at frequency zero.
\begin{definition}
A function $u$ is called a \textit{zero resonant state} for $P$ if $u \in \lecal_0$ is non-zero and $P_0u=0$. If, in addition, $u\in L^2,$ then we call $u$ a \textit{zero eigenfunction}.
\end{definition}
Here, the space $\lecal$ is a variant of the $LE$ space where there is no time dependence (i.e. the time is fixed, and there is no time derivative arising in the norm), and $\lecal_0$ is the closure of $C_c^\infty$ in the $\lecal$ norm.

For a general wave operator $P$, such resonant states are annihilated by $P$ while having finite energy. However, they also possess an infinite $LE^1$ norm when integrating in $t$ over $[0,\infty)$, which violates local energy decay. Such states are ruled out in the context of this paper due to the uniform ellipticity of $P_0$. A quantitative condition on the existence of such resonant states is as follows.
\begin{definition}
$P$ is said to satisfy a \textit{zero resolvent bound/zero non-resonance condition} if there exists some $K_0$, independent of $t$, such that 
\begin{align}\label{zero res}
    \norm{u}_{\dot{H}^1}\leq K_0\norm{P_0 u}_{\dot{H}^{-1}}\qquad \forall u\in \dot{H}^1.
\end{align} 
\end{definition}
Proposition 2.10 of \cites{mst20} demonstrates that a stationary wave operator $P$ has no zero resonant states/zero eigenfunctions if and only if the zero non-resonance condition holds. In our problem, this condition is satisfied due to the uniform ellipticity of $P_0$. 

The relevant low frequency estimate is the following, and the corresponding theorem in \cites{mst20} is Theorem 6.1.
 \begin{Th}
Let $P$ be an asymptotically flat damped wave operator, and suppose that $\partial_t$ is uniformly time-like. Then,
\begin{equation}\label{low freq est}
\norm{u}_{LE^1}\lesssim\norm{\partial_t u}_{LE^1_c}+\norm{Pu}_{LE^*}
\end{equation} for all $u\in \mathcal{S}(\R^4)$.
\end{Th}
 \begin{remark}\label{low freq rem}
The error term $\norm{\partial_t u}_{LE^1_c}$ has the unfortunate effect of requiring information on the size of first-order derivatives of $\partial_t u.$ However, this estimate will only be used when the time frequency is close to zero, in which case this term will be absorbable into the left-hand side of the inequality. Indeed, if we consider $u\in\mathcal{S}(\R^4)$ with frequency support $0\leq |\tau|\leq \tau_0\ll 1$, then we may apply Plancherel's theorem to obtain that
$$\norm{\partial_t u}_{LE^1_c}\lesssim \tau_0\norm{u}_{LE^1_c}.$$ If $\tau_0$ is sufficiently small, then we may absorb this term into the lower-bound side of (\ref{low freq est}) to obtain local energy decay for such $u$.
\end{remark} 

The proof leverages weighted elliptic estimates for the flat Laplacian $\Delta$ in order to get similar estimates for $AF$ perturbations. Once again, the damping does not play a harmful (or even meaningful) role. At frequency zero, it provides no contribution, and near frequency zero, it can be readily absorbed by the error term in (\ref{low freq est}).

\subsection{Establishing Local Energy Decay}\label{ext arg}
Now, we discuss the second main theorem, local energy decay. 
First, the authors in \cites{mst20} show that it is sufficient to remove the Cauchy data at times $0$ and $T$. This makes it significantly easier to perform frequency localization. The corresponding theorem in \cites{mst20} is Theorem 7.1.
\begin{Th}\label{uncond LED}
Let $P$ be a stationary, asymptotically flat damped wave operator satisfying the geometric control condition (\ref{GCC}), and suppose that $\partial_t$ is uniformly time-like. Then, the estimate
\begin{align}
    \label{uncond LED est}
    \norm{u}_{LE^1}\lesssim\norm{Pu}_{LE^*}
\end{align} holds for all $u\in\mathcal{S}(\R^4)$. 
\end{Th}
 In order to prove Theorem \ref{uncond LED}, one splits $u$ into its low, medium, and high frequency parts using a time-frequency partition of unity. In each relevant frequency regime, one applies the corresponding frequency estimate (as in Remarks \ref{high freq rem}, \ref{med freq rem}, and \ref{low freq rem}), and then sums them together. For the medium frequency estimate to be compatible with the low and high frequency regimes, one needs that it apply to any range of time frequencies bounded away from both zero (compatibility with low) and infinity (compatibility with high), which is the utility of the $\delta$ parameter in Theorem \ref{med freq est}. The commutators of the time-frequency cutoffs and $P$ are zero since $P$ is stationary.

 As in Section 7 of \cites{mst20}, one proves that Theorem \ref{uncond LED} implies Theorem \ref{LED thm} by fixing $u$ and constructing a function $v$ which matches the Cauchy data of $u$ at times $0$ and $T$ (and satisfies an appropriate bound) which allows one to apply (\ref{uncond LED est}) to $u-v.$ This construction is performed using a partition of unity on the support of $u[0]$, $u[T]$, and $Pu.$ In particular, one splits into an interior region $\{|x|<4R_0\}$ and an exterior region $\{|x|>2R_0\}$. The damping is non-problematic in the interior (here, one uses the uniform energy bounds) and is zero in the exterior. It is important to highlight the latter fact since the authors use a time reversal symmetry argument in the exterior region, and time reversal turns the damping into a driving force (and hence a harmful term). However, our damping is zero in the exterior region, rendering such an argument non-problematic by choosing an appropriate small $AF$ perturbation of $\Box$ which matches $P$ in the exterior (where the damping is zero). In the context of \cites{mst20}, this provides the ``two point" local energy estimate $$\norm{u}_{LE^1[0,T]}+\norm{\partial u}_{L^\infty_t L^2_x[0,T]}\lesssim\norm{\partial u(0)}_{L^2}+\norm{\partial u(T)}_{L^2}+\norm{Pu}_{LE^*+L^1_tL^2_x[0,T]}. $$ In view of Corollary \ref{uniform ee}, this implies local energy decay.
\bibliographystyle{amsplain}
\bibliography{main}
\end{document}